\documentclass{article}
\usepackage{indentfirst}
\usepackage{amsmath}
\usepackage{amssymb}
\usepackage{amsthm}
\usepackage{authblk}
\usepackage{ragged2e}
\usepackage{paracol}
\usepackage{geometry}
\newtheorem{theorem}{Theorem}[section]
\newtheorem{lamme}{Lemma}[section]
\newtheorem{definition}{Definition}[section]
\newcommand\keywords[1]{\textbf{Keywords}: #1}
\title{Finite-time blow-up in a class of chemotaxis systems with spatially heterogeneous diffusion sensitivity}

\author[1]{Yashuang Zhao}
\author[1]{Shijun Li}
\author[1]{Shaopeng Xu \thanks{Corresponding author: xuxsp@126.com}}
\affil[1]{\raggedright College of Mathematics and Statistics, Hainan University, Haikou 570228, China}
\date{}

\begin{document}
	\maketitle
	\begin{abstract}
		\numberwithin{equation}{section}
		\indent In this paper, we study a class of parabolic-elliptic Keller-Segel systems with diffusion sensitivity dependent on spatial position, given by type
		\begin{equation}
			\left\{    
			\begin{array}{ll}
				u_{t} = \bigtriangledown\cdot(|x|^{\beta} \bigtriangledown u)-\bigtriangledown\cdot(u^{\alpha} \bigtriangledown v),\\
				0=\bigtriangleup v-\mu +u, \qquad \mu:=\frac{1}{|\Omega|}\int_{\Omega}udx,\\
			\end{array}
			\right.
		\end{equation}
		under homogeneous Neumann conditions in a ball $\Omega=B_{R}(0)\subset \mathbb{R}^{n}$ with $\alpha \ge 1$, $\beta>0$ and $n\ge 2$.\par
		\indent It is proved that any nonconstant nonnegative radial initial data $u_{0}\in C^{\theta}(\overline{\Omega})$, where $\theta \in (0,1)$, there exists a radially symmetric classical solution of the system (0.1) in $(\Omega \setminus \{ 0 \})\times (0,T)$ for some $T>0$; moreover, if the initial values $u_{0}\in C^{1+\theta}(\overline{\Omega})$ for some $\theta \in (0,1)$ and satisfy a certain compatibility criterion and are radially decreasing, then this solution is bounded and unique in $(\Omega \setminus \{ 0 \})\times (0,T^{*})$ with $T^{*}<T$.\par
		Finally, it is found that the initial mass corresponding to this parabolic-elliptic problem (0.1) is sufficiently concentrated to allow the solution to blow up in finite time.\\\\
	\end{abstract}
	
	\section{Introduction}	
		\indent From tiny bacteria to the largest mammals, the survival and livelihood of many organisms depends on their ability to navigate in complex environments. This includes detecting, integrating, and processing a variety of internal and external signals. This movement influenced by the concentration of chemical signals is defined as chemotaxis-the directed movement of a biological cell or organism in response to a chemical gradient.\par
			
		\indent Theoretical studies of chemotaxis can be traced back to Patlak's pioneering work in the 1950s \cite{1953Patlak}. Subsequently, Keller and Segel derived a mathematical model of chemotaxis in the 1970s through their work on the slime mold amoeba \cite{KELLER1971225}, and as predicted at the time, a distinguishing feature of the Keller-Segel system was that they described the potential for spontaneous emergence of cell aggregation in terms of the process of singularity formation in the sense of mathematical extremes. The manifestation of chemotaxis is omnipresent in life. For instance, chemotaxis attracts immune cells to migrate to the site of inflammation and directs fibroblasts into the injured area to promote wound healing \cite{Wu2005SignalingMF}; fruit flies navigate along the gradients of attractive odors during food localization \cite{PMID:16857884}; male moths follow the pheromone gradient released by females during mating orientation \cite{Kennedy1974PheromoneRegulatedAI}; vascular endothelial growth factor stimulates new blood vessel growth \cite{article2000}.\par
		
		\indent This work considers the initial-boundary value problem
		\begin{equation}
			\left\{    
			\begin{array}{lll}
				u_{t} = \bigtriangledown\cdot(|x|^{\beta} \bigtriangledown 	u)-\bigtriangledown\cdot(u^{\alpha} \bigtriangledown v), & x\in \Omega, & t>0,\\
				0=\bigtriangleup v-\mu +u, & x\in \Omega, & t>0,\\
				\frac{\partial u}{\partial\nu}=\frac{\partial v}{\partial\nu}=0, & x\in 	\partial \Omega, & t>0,\\
				u(x,0)=u_{0}(x), & x\in \Omega,\\
				\int_{\Omega} v(x,t) dx=0, & t>0,
			\end{array}
			\right.  
		\end{equation}
		where $\Omega = B_{R}(0) \subset \mathbb{R}^{n}$ is a ball, $n \ge 2$, $R>0$, $\beta > 0$ and $\alpha \ge 1$. The initial value 
		\begin{equation}
			u_{0}\in C_{rad}^{0}(\overline{\Omega}):=\Big \lbrace \varphi \in C^{0}(\overline{\Omega}) \mid \varphi \ is \ radially \ symmetric \Big \rbrace
		\end{equation}
		is assumed to be a nonnegative function with $u_{0} \not\equiv 0$ and 
		$$
			\mu :=\frac{1}{|\Omega|}\int_{\Omega}udx
		$$
		represents the spatial average of cell density.	
		
		\indent In $(1.1)$, $u=u(x,t)$ and $v=v(x,t)$ denote the difference between cell density and chemical signal concentration and their spatially averaged concentration values, respectively. The classical parabolic-parabolic chemotaxis model proposed by Keller and Segel in the 1970s has been extensively studied in the past decades. For $n = 1$, Osaki and Yagi derived the existence of globally bounded classical solutions \cite{2001Finite}. For n = 2, Nagai et al concluded that solution $(u,v)$ exists globally and is globally bounded if the radial function $u_{0}(x)$ satisfies condition $\int_{\Omega}u_{0}(x)dx<8\pi$; solution $(u,v)$ blow up in finite time if the radial function $u_{0}(x)$ satisfies condition $\int_{\Omega}u_{0}(x)dx>8\pi$ and $\int_{\Omega}u_{0}(x)|x|^{2}$ is sufficiently small \cite{1997Application}. For $n\ge 3$, Winkler concluded that the solution must blow up in finite time \cite{Winkler2010AggregationVG}. More discussion of the form of classical model solutions can be found in review papers \cite{Bellomo2015TowardAM, Hillen2009AUG, 2003From, nagai2001blowup, herrero1997blow, horstmann2001blow, winkler2010does, horstmann2005boundedness}.
		
		\indent For the second equation of the system (1.1), based on the fact that in many cases the diffusion of chemicals is much faster than the movement of cells, in \cite{jager1992explosions} equation $0=\bigtriangleup v-\mu +u$ was derived as an asymptotic limit that is still biologically meaningful, where $\mu$ denotes the spatial average of cell density, $v$ denotes the difference between the chemical signal concentration and its spatial average concentration and satisfies the property $\int_{\Omega}v(x,t)dx=0$. This type of Keller-Segel system has also been studied with very rich findings. In the literature \cite{winkler2019unstable}, the parabolic-elliptic Keller-Segel system of equations
		$$
			\left\{    
			\begin{array}{lll}
				u_{t} = \bigtriangledown\cdot(\bigtriangledown 	u)-\bigtriangledown\cdot(u\bigtriangledown v), & x\in \Omega, & t>0,\\
				0=\bigtriangleup v-\mu +u, & x\in \Omega, & t>0,
			\end{array}
			\right.
		$$
		is considered in ball $\Omega=B_{R}(0)\subset \mathbb{R}^{n}$. The model is a simplified parabolic-elliptic version of the classical Keller-Segel model. It shows that for any choice of $n\ge 2$ and $R>0$, there exists a positive number $m_{c}$ with $\int_{\Omega}u_{0}dx=m$ for any nonnegative radial initial data $u_{0}$ in a suitably defined sense, and that this initial value problem allows solutions to blow up in finite time if $m>m_{c}$ and $u_{0} \ge \frac{m}{|\Omega|}$ but $u_{0} \not\equiv \frac{m}{|\Omega|}$; however, if $m<m_{c}$, there exists nonnegative radial function $u_{0}$ that satisfy $\int_{\Omega}u_{0}dx=m$ and is more centralized than the $u\equiv \frac{m}{|\Omega|}$, but still allow solution emanating from $u_{0}$ globally bounded. In the literature \cite{winkler2010boundedness}, Winkler and Djie studied the elliptic-parabolic Keller-Segel model
		$$
			\left\{    
			\begin{array}{lll}
				u_{t} = \bigtriangledown\cdot(\phi (u) \bigtriangledown 	u)-\bigtriangledown\cdot(\psi (u) \bigtriangledown v), & x\in \Omega, & t>0,\\
				0=\bigtriangleup v-\mu +u, & x\in \Omega, & t>0,
			\end{array}
			\right.
		$$
		where $\phi (u)= (u+1)^{-p}$, $\psi (u)= u(u+1)^{q-1}$ with some $p \ge 0$ and $q \in \mathbb{R}$. It is obtained that when $p+q<\frac{2}{n}$, all solutions are globally bounded in time; and when $p+q>\frac{2}{n}$, $q>0$ and $\Omega$ is a ball, the solutions blow up in finite time. Furthermore, in 2022, Winkler studied the proposed linear Keller-Segel model
		$$
		\left\{    
		\begin{array}{lll}
			u_{t} = \bigtriangledown\cdot(\phi (u) \bigtriangledown 	u)-\bigtriangledown\cdot(u\psi (u) \bigtriangledown v), & x\in \Omega, & t>0,\\
			0=\bigtriangleup v-\mu +u, & x\in \Omega, & t>0.
		\end{array}
		\right.
		$$
		When the region is a ball, the no-flux boundary is considered, the dimension $n\ge 2$ is obtained in the literature \cite{winkler2024slow}, if there exists certain constants $K>0$ and $\lambda > \frac{2}{n}$ such that $\frac{\xi \psi(\xi)}{\phi(\xi)}\ge K \xi^{\lambda}$, where $\phi$ and $\psi$ satisfy certain additional technical conditions, then the corresponding global solution blow up at infinity. Moreover, when the dimension $n\ge 3$, Winkler obtained in the literature \cite{winkler2024complete} that the equations have a unique global classical solution, where $\phi$ and $\psi$ satisfy some further technical conditions. Similar parabolic-elliptic models have been studied and discussed by Xinyu Tu, Bellomo, Chiyoda et al \cite{tu2022effects, bellomo2017finite, chiyoda2020finite, cieslak2010finite}.
					
		\indent The motion of the system is compromised near the origin, i.e. the diffusion rate can be expressed in the form of $|x|^{\beta} (\beta>0)$. A more intuitive illustration is that the motility of a cell or bacterial population is compromised near the origin when a clotting mechanism takes place, which has long been regarded as significant in medicine, particularly for halting bleeding. Keller-Segel systems with spatially heterogeneous diffusion sensitivity were initially mentioned in \cite{fluchter2024solutions}. This type of model is denoted as
		$$
			\left\{    
			\begin{array}{lll}
				u_{t} = \bigtriangledown\cdot(|x|^{\beta} \bigtriangledown 	u)-\bigtriangledown\cdot(u \bigtriangledown v), & x\in \Omega, & t>0,\\
				0=\bigtriangleup v-\mu +u, & x\in \Omega, & t>0,
			\end{array}
			\right.
		$$ 
		in \cite{fluchter2024solutions}. Flüchter considering the form of the cross-diffusion term $-div(u\bigtriangledown v)$, concluded that the system has a radially symmetric pointwise classical solution when the dimension $n \ge 1$. Whereas a unique bounded classical solution exists when the dimension $n\ge 2$, the initial data are radially decreasing and some further technical conditions are satisfied. When the initial data are sufficiently accumulated, there are no globally bounded radially symmetric solutions in time.
		
		To solve this type of parabolic elliptic problem, since the Lyapunov function corresponding to problem (1.1) is still an intractable problem, we refer to the method in the literature \cite{winkler2011blow, cieslak2008finite}, which employs the mass distribution function
		$$
		     w(s,t)=\int_{0}^{s^{\frac{1}{n}}}\rho^{n-1}u(\rho, t)d\rho, \qquad s\in [0,R^{n}], \ t\ge 0.
		$$
		This method can be regarded as an ODI analysis of some $L^{p}$ seminorms with singular weights and p smaller than one. In fact, $w$ satisfies the degenerate scalar parabolic equation and the solution of the research problem (1.1) is the research scalar parabolic equation.\par
		
		The difference with the system studied in the literature \cite{fluchter2024solutions} is that the chemotaxis term in this study is of the form $-div(u^{\alpha}\bigtriangledown v)$, which makes the method of dealing with the quasi-linear parabolic equations in the literature \cite{fluchter2024solutions} inapplicable, and the discrepancy is mainly due to the fact that $\alpha > 2$. In order to be able to solve this intractable problem, this paper constructs a bounded convex closed set to obtain the local existence of the solution through Schauder fixed point theorem.
		
		\indent Our main result is in a framework based on the following classical solution definition:
		\begin{definition}
			Let $n\ge 1$, $R>0$, $\Omega=B_{R}(0)\subset \mathbb{R}^{n}$, $T\in (0,\infty]$, $\Omega_{0}:=\overline{\Omega} \setminus \{ 0 \}$, and $u_{0}$ conform to $(1.2)$. Then, we call a pair of function $(u,v)$ satisfying
			\begin{equation}
				\left\{    
				\begin{array}{lll}
					u\in C^{0}(\Omega_{0}\times [0,T))\cap C^{2,1}(\Omega_{0}\times (0,T)),\\
					v\in C^{2,0}(\Omega_{0}\times (0,T)),
				\end{array}
				\right.
			\end{equation}
			and solving $(1.1)$ in $\Omega_{0}\times [0,T)$ a classical solution of $(1.1)$ in $\Omega_{0} \times [0,T)$ that
			$$
				if \ T<\infty, \ then \ \lim_{t \nearrow T}sup \Vert u(\cdot, t) \Vert_{L^{\infty}(\Omega_{0})}=\infty.
			$$
			Moreover, we call $u(\cdot,t)$ and $v(\cdot,t)$ are radially symmetric in $(0,R]\times [0,T)$ satisfies the following properties
			$$
				\left\{    
				\begin{array}{lll}
					u\in C^{0}((0, R]\times [0,T))\cap C^{2,1}((0, R]\times (0,T)),\\
					v\in C^{2,0}((0, R]\times (0,T)),
				\end{array}
				\right.
			$$
			is a solution of (2.1).
		\end{definition}
	
	    \noindent \textbf{Main results}$\quad$ In the ball $\Omega=B_{R}(0)\subset \mathbb{R}^{n} (n\ge 2)$ with $R>0$, if $\beta >0$, $\alpha \ge 1$ and initial values in system (1.1) satisfy (1.2).\par 	
	    \columnratio{0.04}
	    \begin{paracol}{2}
	    	$\bullet$
	    	\switchcolumn
	    	\noindent For some $\theta \in (0,1)$, suppose initial value $u_{0}\in C^{\theta}(\overline{\Omega})$, then there exists a radially symmetric classical solution defined by Definition 1.1 in $(\Omega_{0}\times (0,T_{0}))$ for some $T_{0}\in (0,\infty]$.
	    \end{paracol}
	    \begin{paracol}{2}
	    	$\bullet$
	    	\switchcolumn
	    	\noindent For some $\theta \in (0,1)$, suppose that initial value $u_{0}\in C^{1+\theta}(\overline{\Omega})$ is radially decreasing, with the additional requirements
	    	$$
	    	u_{0}=0\qquad and \qquad \bigtriangledown u_{0} \cdot \nu =0 \qquad on \ \partial \Omega
	    	$$
	    	as well as
	    	$$
	    	|\bigtriangledown u_{0}(x)|\le c_{0}|x|^{n-1+\theta} \quad for \ some \ c_{0}>0.
	    	$$ 	    
	    	This solution is bounded and unique in $(\Omega_{0}\times (0,T^{*}))$ and satisfies
	    	\begin{equation}
	    		0\le u\in L^{\infty}(\Omega\times(0,T)) \ and \ \bigtriangledown v \in L^{\infty}(\Omega\times(0,T); R^{n})
	    	\end{equation}
	    	for all $T\in (0,T^{*})$ with $T^{*}\in (0,T_{0}]$.
	    \end{paracol}
	    \begin{paracol}{2}
	    	$\bullet$
	    	\switchcolumn
	    	\noindent The system (1.1) defines $m:=\int_{\Omega}u_{0}$ and each $m_{0}\in (0,m]$, if there exists $r_{0}=r_{0}(m_{0}, m, R, \beta, \alpha)>0$ such that
	    	$$
	    	\int_{B_{r_{0}}(0)}u_{0} \ge m_{0}
	    	$$ 
	    	holds, and there is no global classical solution of (1.1) satisfying the property (1.4).
	    \end{paracol}
	     
	     \noindent\textbf{Overview of the arguments}$\quad$For the Keller-Siegel problem $(1.1)$ with Neumann boundary, the main idea of this study is that $(1.1)$ is transformed to a hybrid problem by means of the mass distribution function, then prove the existence and uniqueness of the local solution of the problem by constructing the approximation problem, Schauder estimates, the comparison principle, etc. (cf. Section 2), and then re-transform it to obtain a solution of the problem by means of a hybrid between the Keller-Segel problem (1.1) and the problem connection between the Keller-Segel problem (1.1) and the hybrid problem in order to obtain a solution to the Keller-Segel problem (1.1) (cf. Section 3). Furthermore, we consider the existence of a classical solution to the problem (1.1) for the function $u$ in global time with a suitable set of initial data, and it turns out that such a solution does not exist, i.e., it blow up in finite time (cf. Section 4).
	     
		\setcounter{equation}{0}
				
	\everymath{\displaystyle}			
	\section{Parabolic problems satisfied by the mass accumulation function}
	We will prove the existence and uniqueness of sufficiently smooth solutions to the parabolic problem satisfying the mass accumulation function.\par
	Referring to the experience of the literature \cite{jager1992explosions}, to consider the mass distribution function, then radial symmetry is important.
	
	\subsection{Transformations of a nonlocal scalar parabolic equation}
	For some $R>0$, we base this on the fact that $\Omega = B_{R}(0)$ is radially symmetric with respect to $x=0$, we may as well assume that $(u, v)$ denotes the corresponding radial solution in $(\overline{\Omega} \setminus \{ 0 \})\times [0,T)$ asserted in Definition 1.1. In what follows, without confusion, we may write $u=u(r, t)$, $v=v(r, t)$, where $r=|x| \in (0,R]$.\par
	We rewrite system $(1.1)$ in radialized form, i.e.,
	\begin{equation}
		\left\{    
		\begin{array}{lll}
			u_{t} = \frac{1}{r^{n-1}}(r^{n-1+\beta}u_{r})_{r} -\frac{1}{r^{n-1}}(r^{n-1}u^{\alpha}v_{r})_{r}, & r\in (0,R), & t>0,\\
			0=\frac{1}{r^{n-1}}(r^{n-1}v_{r})_{r}- \mu +u, & r\in (0,R), & t>0,\\
			u_{r} = v_{r} = 0, & r = R, & t>0,\\
			u(r,0)=u_{0}(r), & r\in (0,R),
		\end{array}
		\right.
	\end{equation} 	
	in $(0,R]\times [0,T)$. Then the radially symmetric solution of system $(2.1)$ corresponding to system $(1.1)$ satisfies 
	\begin{equation}
		\left\{    
		\begin{array}{lll}
			u\in C^{0}((0, R]\times [0,T))\cap C^{2,1}((0, R]\times (0,T)),\\
			v\in C^{2,0}((0, R]\times (0,T)).
		\end{array}
		\right.
	\end{equation}
	From Definition 1.1, we call $(u, v)$ with property (2.2) a classical solution of (2.1) in $(0, R]\times [0, T)$.\par
	For a discussion of the second equation in system (2.1), Flüchter has already studied Lemma 2.1 in detail in the literature \cite{fluchter2024solutions}. Moreover, an expression for $v_{r}$ as such has been proposed and used in the literature \cite{winkler2011blow, cieslak2008finite, tu2022effects}.\par
	Unlike the process of verifying the expression of $v_{r}$ in the literature \cite{winkler2011blow, LIN2018435}, the problem studied in this paper is the same as that of Flüchter when it is considered in the literature \cite{fluchter2024solutions}, where the zeros are singularities.
	\hspace*{\fill} \\
	\begin{lamme}[\cite{fluchter2024solutions}]
		Let $n\ge 2, R>0,$ and $u_{0}$ conform to $(1.2)$. Moreover, let $(u, v)$ with $u\in L^{\infty}((0,R)\times (0,T_{0}))$ for some $T_{0}\in (0,T]$ be a classical solution of $(2.1)$ in $(0,R]\times [0,T_{0})$. Then, if
		\begin{equation}
			v_{r}(r,t)=\frac{1}{r^{n-1}} \Big (\frac{\mu r^{n}}{n}-\int_{0}^{r}\rho ^{n-1}u(\rho, t)d\rho \Big ) \qquad for \ all \ (r,t)\in (0,R]\times [0,T_{0}), 
		\end{equation}
		we have 
		\begin{equation}
			|v_{r}(r,t)|\le Cr \qquad for \ all \ (r,t)\in (0,R]\times [0,T_{0})
		\end{equation}
		with $C:= \frac{2}{n}\parallel u \parallel _{L^{\infty}}((0,R)\times (0,T_{0}))$, and hence in particular $v_{r}\in L^{\infty}((0,R)\times (0,T_{0}))$. \\
		Moreover, the converse proposition also holds:\\
		If $v_{r}\in L^{\infty}((0,R)\times (0,T_{0}))$, then necessarily $(2.3)$.
	\end{lamme}
	Even though $|x|^{\beta}$ is not integrable at $0$ when $0 < \beta \le 1$ and the Laplace coefficient of $u$ is zero at $x = 0$ for all $\beta > 0$, this implies diffusion degeneracy. But under assumptions of $\beta, \ \alpha$ and $n$, the classical solution of (2.1) conserves $\Vert u(\cdot, t) \Vert_{L^{1}((0,R))}$ at least when $u$ satisfies the conditions of Lemma 2.1.
	\hspace*{\fill} \\
	\begin{lamme}
		Let $n\ge 2, R>0, \beta >0, \alpha\ge 1$ and $u_{0}$ conform to $(1.2)$. Assume that $(u, v)$ is a classical solution of $(2.1)$ on the basis of Definition 1.1. Moreover, suppose in addition that
		$$
		    0 \le u\in L^{\infty}((0,R)\times (0,T_{0}))\qquad for \ some \ T_{0}\in (0,T]
		$$
		as well as $(2.3)$ holds. Then the following mass conservation property is satisfied, 
		\begin{equation}
			\int_{0}^{R}\rho ^{n-1}u(\rho, t)d\rho =\int_{0}^{R}\rho ^{n-1}u_{0}(\rho)d\rho \qquad for \ all \ t_{0}\in [0,T).
		\end{equation}
	\end{lamme}
	\begin{proof}
		Let $(\zeta ^{(\delta)})_{\delta \in (0,\frac{R}{2})}$ be a cutoff function such that it has $\zeta ^{(\delta)}\in C^{\infty}([0,R])$ satisfying
		\begin{equation}
			\left\{    
			\begin{array}{ll}
				\zeta ^{\delta}(r)=0, & r\in [0, \frac{\delta}{2}],\\
				0 \le \zeta ^{\delta}(r) \le 1, & r\in (\frac{\delta}{2}, \delta),\\
				\zeta ^{\delta}(r)=1, & r\in [\delta, R],
			\end{array}
			\right.
		\end{equation}
		with
		\begin{equation}
			0\le \zeta _{r} ^{\delta}(r) \le \frac{C_{1}}{\delta}, \qquad r\in (\frac{\delta}{2}, \delta),
		\end{equation}		
		and
		\begin{equation}
			|\zeta^{\delta}_{rr}(r)|\le \frac{C_{2}}{\delta^{2}}, \qquad r\in (\frac{\delta}{2}, \delta),
		\end{equation}
		for all $\delta \in (0, \frac{R}{2})$ and some positive constants $C_{1}, \ C_{2}$ independent of $\delta$.\\
		From $(2.6)$ we know that for all $\delta\in (0, \frac{R}{2})$ and $t\in (0, T_{0})$ we get $\zeta^{(\delta)}u_{t}(\cdot, t)\in L^{1}((0,R))$. Using $(2.1)$ and the partial integral we can calculate
		\begin{equation} 
			\begin{array}{ll}
				\frac{d}{dt}\int_{0}^{R}r^{n-1}\zeta^{(\delta)}udr & = \displaystyle\int_{0}^{R}r^{n-1}\zeta^{(\delta)}u_{t}dr \\
				\quad & =\int_{0}^{R}r^{n-1}\zeta^{(\delta)} \Big\lbrack \frac{1}{r^{n-1}}(r^{n-1+\beta}u_{r})_{r} -\frac{1}{r^{n-1}}(r^{n-1}u^{\alpha}v_{r})_{r} \Big\rbrack \\
				\quad & =\int_{0}^{R}\zeta^{(\delta)}(r^{n-1+\beta}u_{r})_{r}-\zeta^{(\delta)}(r^{n-1}u^{\alpha}v_{r})_{r}\\
				\quad & = \int_{0}^{R} (\zeta^{(\delta)}_{r}r^{n-1+\beta})_{r}udr+\int_{0}^{R}\zeta^{(\delta)}_{r}r^{n-1}u^{\alpha}v_{r}dr.
			\end{array}
		\end{equation}
		Here, assuming $C_{3}:=\Arrowvert u \Arrowvert_{L^{\infty}((0,R)\times (0,T_{0}))}$, we further estimate by $n-2+\beta >0$ that for $t\in (0,T_{0})$ there is 
		\begin{equation}
			\begin{array}{l}
				\Big\vert \int_{0}^{R} (\zeta^{(\delta)}_{r}r^{n-1+\beta})_{r}udr \Big\vert\\
				\qquad\qquad \le (n-1+\beta)\int_{\frac{\delta}{2}}^{\delta}\Big\vert r^{n-2+\beta}\zeta^{(\delta)}_{r}u \Big\vert +\int_{\frac{\delta}{2}}^{\delta}\Big\vert r^{n-1+\beta}\zeta^{(\delta)}_{rr}u \Big\vert \\
				\qquad\qquad \le (n-1+\beta)C_{3}\int_{\frac{\delta}{2}}^{\delta}\Big\vert r^{n-2+\beta}\zeta^{(\delta)}_{r} \Big\vert +C_{3}\int_{\frac{\delta}{2}}^{\delta}\Big\vert r^{n-1+\beta}\zeta^{(\delta)}_{rr} \Big\vert \\
				\qquad\qquad \le C_{1}C_{3}\delta^{n-2+\beta}+\frac{1}{n+\beta}C_{2}C_{3}\delta^{n-2+\beta}\to 0
			\end{array}
		\end{equation} 
		at $\delta \searrow 0$.\\
		Due to $(2.3)$, Lemma 2.1 guarantees that for some positive constant $C$ there is
		$$
			|v_{r}(r,t)|\le Cr \qquad for \ all \ (r,t)\in (0,R]\times(0,t_{0}), 
		$$
		and thus since $n\ge 2$, for $\delta \searrow 0$ we can compute
		\begin{equation}
			\begin{array}{ll}
				\Big \vert \int_{0}^{R}\zeta^{(\delta)}_{r}r^{n-1}u^{\alpha}v_{r}dr \Big \vert & \le C \int_{0}^{R}\zeta^{(\delta)}_{r}r^{n}|u^{\alpha}|dr\\
				\qquad & \le C \Arrowvert u^{\alpha} \Arrowvert_{L^{\infty}}\int_{0}^{R}\zeta^{(\delta)}_{r}r^{n}dr\\
				\qquad & \le \frac{CC_{1}C_{3}^{\alpha}}{n+1}\delta^{n}\to 0.
			\end{array}
		\end{equation}
		For $t\in (0,T_{0})$, integrating $(2.9)$ over $(0,t)$ yields
		$$
			\begin{array}{l}
				\int_{0}^{R}r^{n-1}\zeta^{(\delta)}u(r,t)dr-\int_{0}^{R}r^{n-1}\zeta^{(\delta)}u_{0}(r)dr \\
				\qquad\qquad=\int_{0}^{t}\int_{0}^{R} (\zeta^{(\delta)}_{r}r^{n-1+\beta})_{r}udrdt+\int_{0}^{t}\int_{0}^{R}\zeta^{(\delta)}_{r}r^{n-1}u^{\alpha}v_{r}drdt.
			\end{array}
		$$
		Based on $(2.10), \ (2.11)$ and dominated convergence and letting $\delta \searrow 0$, then $(2.5)$ is obtained.
	\end{proof}
	 For further analysis, it is of utmost importance to transform $(2.1)$ into the hybrid problem.\par
	 Following \cite{jager1992explosions}, we introduce the mass distribution function
	 \begin{equation}
	 	w(s,t)=\int_{0}^{s^{\frac{1}{n}}}\rho^{n-1}u(\rho, t)d\rho, \qquad s=r^{n}\in [0,R^{n}], \ t\in [0,T).
	 \end{equation}
	Then, we can compute
	\begin{equation}
		w_{s}(s,t)=\frac{1}{n}u(s^{\frac{1}{n}}, t),\qquad w_{ss}(s,t)=\frac{1}{n^{2}}s^{\frac{1}{n}-1}u_{r}(s^{\frac{1}{n}}, t)
	\end{equation}
	for all $s\in (0,R^{n})$ and $t\in (0,T)$.\\
	Since $(2.1)$ and $(2.13)$, 
	$$
		\begin{array}{ll}
			w_{t}(s,t) & =\int_{0}^{s^{\frac{1}{n}}}\rho^{n-1}u_{t}(\rho, t)d\rho\\
			\qquad & =\int_{0}^{s^{\frac{1}{n}}} (\rho^{n-1+\beta}u_{r}(\rho, t))_{\rho}d\rho -\int_{0}^{s^{\frac{1}{n}}}(\rho^{n-1}u^{\alpha}(\rho, t)v_{r}(\rho, t))_{\rho}d\rho \\
			\qquad & = s^{\frac{n-1+\beta}{n}}u_{r}(s^{\frac{1}{n}}, t)-u^{\alpha}(s^{\frac{1}{n}}, t)\Big( \frac{\mu s}{n}-\int_{0}^{s^{\frac{1}{n}}}r^{n-1}u(r, t)dr \Big)\\
			\qquad & =n^{2}s^{\frac{2n-2+\beta}{n}}w_{ss}+n^{\alpha}ww_{s}^{\alpha}-n^{\alpha-1}\mu sw_{s}^{\alpha}
		\end{array}
	$$
	is nonlocally nonlinear for all $s\in (0,R^{n})$ , $t\in (0,T)$.\\
	Therefore, $w\in C^{0}([0,R^{n}]\times [0, T))\cap C^{2,1}((0,R^{n}]\times (0, T))$ solves the hybrid problem 
	\begin{equation}
		\left\{    
		\begin{array}{ll}
			w_{t}=n^{2}s^{\frac{2n-2+\beta}{n}}w_{ss}+n^{\alpha}ww_{s}^{\alpha}-n^{\alpha-1}\mu sw_{s}^{\alpha}, & s\in (0,R^{n}), \ t\in (0,T),\\
			w(0,t)=0, \ w(R^{n}, t)=\frac{m}{\omega_{n}}, & t\in (0,T),\\
			w(s,0)=w_{0}(s), & s\in (0,R^{n}),
		\end{array}
		\right.
	\end{equation}
	where $m:=\int_{\Omega}u_{0}$, $\mu =\frac{nm}{\omega_{n}R_{n}}$ and
	\begin{equation}
		w_{0}(s)=\int_{0}^{s^{\frac{1}{n}}}\rho^{n-1}u_{0}(\rho)d\rho, \qquad s\in [0,R^{n}].
	\end{equation}
	\indent It follows that there exists a bounded classical solution $w$ to the mixed problem $(2.14)$ if $(u,v)$ is a bounded classical solution of $(2.1)$ with property $(2.3)$. Therefore, to study the solution of the former problem, it seems sensible to study the latter system.
	
	\subsection{Approximation problem for the system (2.14)}
	It is clear that the hybrid problem $(2.14)$ contains diffusion degeneracy and is a quasi-linear problem. However, similar problems utilizing approximation problems and standard theoretical solvability have been treated in the literature \cite{fluchter2024solutions}. Moreover, examples of possible simplicial mergers that depend even on $u$ and $v$ are discussed in the literature \cite{winkler2022approaching, winkler2024bounds}. However, in the considerations of this paper, $\alpha >2$ is an open problem.\par
	Therefore refer to the solution in the literature \cite{fluchter2024solutions} to construct its approximation problem 
	\begin{equation}
		\left\{    
		\begin{array}{ll}
			w_{\varepsilon t}=n^{2}s^{\frac{2n-2+\beta}{n}}w_{\varepsilon ss}+n^{\alpha}w_{\varepsilon}w_{\varepsilon s}^{\alpha}-n^{\alpha-1}\mu (s-\varepsilon)w_{\varepsilon s}^{\alpha}, & s\in (\varepsilon,R^{n}), \ t>0,\\
			w_{\varepsilon}(\varepsilon,t)=0, \ w_{\varepsilon}(R^{n}, t)=\frac{m}{\omega_{n}}, & t>0,\\
			w_{\varepsilon}(s,0)=w_{0\varepsilon}(s), & s\in (\varepsilon,R^{n}),
		\end{array}
		\right.
	\end{equation}
	where $\varepsilon \in (0,R^{n})$ and
	\begin{equation}
		w_{0\varepsilon}(s)=w_{0}\Big(\frac{R^{n}(s-\varepsilon)}{R^{n}-\varepsilon}\Big), \quad s \in [\varepsilon, R^{n}].
	\end{equation}
	This will play a crucial role in the following study of the hybrid problem role.\par
	In what follows, we always let $m:=\int_{\Omega}u_{0} >0$, interpreted as the total mass of the original system. We can obtain the solution of the following degenerate problem $(2.14)$ in its incomplete form:
	\hspace*{\fill} \\
	\begin{lamme}
		Let $n\ge 2$, $R>0$, $\theta \in (0,1)$, $\mu =\frac{nm}{\omega_{n}R^{n}}$, $\beta >0$, $\alpha \ge 1$ and suppose $w_{0}\in C^{1+\theta}([0,R^{n}])$ is such that
		\begin{equation}
			w_{0}(0)=0, \ w_{0}(R^{n})=\frac{m}{\omega_{n}} \ as \ well \ as \ w_{0s}\ge 0.
		\end{equation}
		Then there exists a function $w\in C_{loc}^{1,\frac{1}{2}}((0,R^{n}]\times [0, \infty))\cap C_{ioc}^{2,1}((0,R^{n}]\times (0, \infty))$ which solves the problem
		\begin{equation}
			\left\{    
			\begin{array}{ll}
				w_{t}=n^{2}s^{\frac{2n-2+\beta}{n}}w_{ss}+n^{\alpha}ww_{s}^{\alpha}-n^{\alpha-1}\mu sw_{s}^{\alpha}, & s\in (0,R^{n}), \ t>0,\\
				w(R^{n}, t)=\frac{m}{\omega_{n}}, & t>0,\\
				w(s,0)=w_{0}(s), & s\in (0,R^{n}).
			\end{array}
			\right.
		\end{equation}
		Moreover, w satisfies
		$$
			0\le w(s,t)\le \frac{m}{\omega_{n}}\qquad for \ all \ (s,t)\in (0,R^{n}]\times [0, \infty)
		$$
		and $w(\cdot, t)$ is monotonically increasing in $(0,R^{n}]$ for all $t\in [0, \infty)$. Thus, $w(\cdot, t)$ can be continuously extended to $s=0$ if $t\ge 0$ is fixed, i.e.
		$$
			w(0,t):=\lim_{s \searrow 0} w(s, t)\ge 0.
		$$
	\end{lamme}
	\begin{proof}
		Since $(2.17)$ guarantees $w_{0\varepsilon}(\varepsilon)=0$, $w_{0\varepsilon}(R^{n})=\frac{m}{\omega_{n}}$, $w_{0\varepsilon s}(s)\ge 0$ in $[\varepsilon,R^{n}]$ and $w_{0\varepsilon}(s)\le w_{0}(s)$ in $[\varepsilon, R^{n}]$, $w_{0 \varepsilon}\nearrow w_{0}$ in $(0,R^{n}]$ and $w_{0 \varepsilon}\to w_{0}$ in $C_{loc}^{1}((0,R^{n}])$ as $\varepsilon \searrow 0$. \par
		Define
		$$
			\widetilde{w}_{\varepsilon}(s,t)=\overline{w}_{0\varepsilon}(s)+\int_{0}^{t} n^{2}s^{\frac{2n-2+\beta}{n}}\overline{w}_{\varepsilon ss}+n^{\alpha}\overline{w}_{\varepsilon}\overline{w}_{\varepsilon s}^{\alpha}-n^{\alpha-1}\mu (s-\varepsilon)\overline{w}_{\varepsilon s}^{\alpha} dt,
		$$
		where $\widetilde{w}_{\varepsilon}(\varepsilon,t)=0$, $\widetilde{w}_{\varepsilon}(R^{n}, t)=\frac{m}{\omega_{n}}$ and $\widetilde{w}_{0\varepsilon}(s)=w_{0}\Big(\frac{R^{n}(s-\varepsilon)}{R^{n}-\varepsilon}\Big)$. \par
		Define a bounded convex closed set $A:=\lbrace \overline{w}_{\varepsilon}\in \chi \ \ and \ \Vert \overline{w}_{\varepsilon} \Vert_{\chi}\le C_{m} \rbrace$, where $\frac{1}{2}C_{m}\ge \frac{m}{\omega_{n}}$ and 
		$$
			\chi :=C^{0}([\varepsilon,R^{n}]\times [0, T])\cap C^{2,1}([\varepsilon,R^{n}]\times (0, T]),
		$$
		as well as
		$$
			\Vert \overline{w}_{\varepsilon} \Vert_{\chi}=\Vert \overline{w}_{\varepsilon} \Vert_{C^{0}([\varepsilon,R^{n}]\times [0,T])}+\Vert \overline{w}_{\varepsilon} \Vert_{C^{2,1}([\varepsilon,R^{n}]\times (0,T])},
		$$
		where $T$ is to be determined.\par
		We consider the mapping $F\overline{w}_{\varepsilon} :=\widetilde{w}_{\varepsilon}$, where $\overline{w}_{\varepsilon}\in A$ and $\widetilde{w}_{\varepsilon}$ is the solution of $(2.16)$. Then
		$$   
			\begin{array}{ll}
				\Vert \widetilde{w}_{\varepsilon} \Vert_{\chi} &\le \Big\Vert \overline{w}_{0\varepsilon} + \int_{0}^{t} n^{2}s^{\frac{2n-2+\beta}{n}}|\overline{w}_{\varepsilon ss}|+n^{\alpha}|\overline{w}_{\varepsilon}||\overline{w}_{\varepsilon s}^{\alpha}|+n^{\alpha-1}\mu (s-\varepsilon)|\overline{w}_{\varepsilon s}^{\alpha}| dt \Big\Vert_{\chi}\\
				\qquad &\le \frac{1}{2}C_{m}+n^{2}R^{2n-2+\beta}C_{m}T+n^{\alpha}C^{\alpha+1}_{m}T+n^{\alpha-1}\mu R^{n}C^{\alpha}_{m}T\\
				\qquad &\le C_{m},
			\end{array}
		$$
		where a sufficiently small $T=\frac{1}{2(n^{2}R^{2n-2+\beta}+n^{\alpha}C^{\alpha}_{m}+n^{\alpha-1}\mu R^{n}C^{\alpha-1}_{m})}$. Thus, the operator $F$ is a mapping from $A$ to itself and is uniformly bounded on $A$.\par
		Moreover, for arbitrary $\delta=\frac{\varepsilon}{3(n^{2}R^{2n-2+\beta} T+ n^{\alpha} C^{\alpha}_{m}T +  \alpha(n^{\alpha}C_{m}+n^{\alpha-1}\mu R^{n})C^{\alpha-1}_{m})}> 0$ and $\overline{w}_{\varepsilon 1}, \ \overline{w}_{\varepsilon 2}\in \chi$ satisfies $\Vert \overline{w}_{\varepsilon 1}-\overline{w}_{\varepsilon 2} \Vert_{\chi}<\delta$, by the mean value theorem we can obtain
		$$
			\begin{array}{ll}
				| \widetilde{w}_{\varepsilon 1}(s,t)-\widetilde{w}_{\varepsilon 2}(s,t) | &\le \int_{0}^{t} n^{2}s^{\frac{2n-2+\beta}{n}}|\overline{w}_{\varepsilon 1 ss}-\overline{w}_{\varepsilon 2 ss}| dt+ \int_{0}^{t} n^{\alpha}|\overline{w}_{\varepsilon 1}\overline{w}_{\varepsilon 1 s}^{\alpha}-\overline{w}_{\varepsilon 2}\overline{w}_{\varepsilon 2 s}^{\alpha}| dt\\ 
				\qquad & \qquad +\int_{0}^{t} n^{\alpha-1}\mu (s-\varepsilon)|\overline{w}_{\varepsilon 1 s}^{\alpha}-\overline{w}_{\varepsilon 2 s}^{\alpha}| dt\\
				\qquad &\le n^{2}R^{2n-2+\beta}\delta T+ \int_{0}^{t} n^{\alpha}|\overline{w}_{\varepsilon 1}-\overline{w}_{\varepsilon 2}||\overline{w}_{\varepsilon 1 s}^{\alpha}| dt\\
				\qquad & \qquad +\int_{0}^{t} (n^{\alpha}|\overline{w}_{\varepsilon 2}|+n^{\alpha-1}\mu R^{n})|\overline{w}_{\varepsilon 1 s}^{\alpha}-\overline{w}_{\varepsilon 2 s}^{\alpha}| dt\\
				\qquad &\le n^{2}R^{2n-2+\beta}\delta T+ n^{\alpha}\delta C^{\alpha}_{m}T +  \alpha(n^{\alpha}C_{m}+n^{\alpha-1}\mu R^{n})C^{\alpha-1}_{m}\delta,\\
				\qquad &\le \varepsilon.
			\end{array}
		$$
		Thus, Schauder fixed point theorem guarantees that there exists an element of $A$ satisfying $F\overline{w}_{\varepsilon} :=\overline{w}_{\varepsilon}$. Then $(2.16)$ possesses a classical solution $\overline{w}_{\varepsilon}\in C^{0}([\varepsilon,R^{n}]\times [0, T])\cap C^{2,1}([\varepsilon,R^{n}]\times (0, T])$. To obtain a global solution, we extend problem $(2.16)$ as follows, with 
		$$
			\left\{    
			\begin{array}{ll}
				\widehat{w}_{\varepsilon t}=n^{2}s^{\frac{2n-2+\beta}{n}}\widehat{w}_{\varepsilon ss}+n^{\alpha}\widehat{w}_{\varepsilon}\widehat{w}_{\varepsilon s}^{\alpha}-n^{\alpha-1}\mu (s-\varepsilon)\widehat{w}_{\varepsilon s}^{\alpha}, & s\in (\varepsilon,R^{n}), \ t>T,\\
				\widehat{w}_{\varepsilon}(\varepsilon,t)=0, \ \widehat{w}_{\varepsilon}(R^{n}, t)=\frac{m}{\omega_{n}}, & t>T,\\
				\widehat{w}_{\varepsilon}(s,T)=w_{\varepsilon}(s,T), & s\in (\varepsilon,R^{n}).
			\end{array}
			\right.
		$$
		Obviously, the Schuader fixed point theorem can be used once again for the postponed problem as above, and so on over and over again, we can obtain that there exists a solution $w_{\varepsilon}\in C^{0}([\varepsilon,R^{n}]\times [0, \infty))\cap C^{2,1}([\varepsilon,R^{n}]\times (0, \infty))$ to problem (2.16).\par
		Using Lemma 5.1 gives $0\le w_{\varepsilon}\le \frac{m}{\omega_{n}}$ in $[\varepsilon,R^{n}]\times [0, \infty)$. Since $(w_{\varepsilon})_{\varepsilon \in (0, R^{n})}$ is monotically decreasing in $\varepsilon$, $(w_{\varepsilon})_{\varepsilon \in (0, R^{n})}$ satisfies $w_{\varepsilon}\nearrow w$ in $(0,R^{n}]\times [0, \infty)$ as $\varepsilon\searrow 0$ with some limit function $w$ fulfilling $0\le w\le \frac{m}{\omega_{n}}$. Because $w_{0\varepsilon s}\ge 0$ and $w_{\varepsilon s}(s)\ge 0$ both on $\{\varepsilon\}\times (0, \infty)$ and on $\{R^{n}\}\times (0, \infty)$, we may obtain $w_{\varepsilon s}(s)\ge 0$ in $(\varepsilon, R^{n})\times (0, \infty)$. Then, based on interior parabolic Schauder estimates and the Arzelà-Ascoli theorem, we can further obtain that $w_{\varepsilon}\to w$ in $C_{loc}^{1,\frac{1}{2}}((0,R^{n}]\times [0, \infty))$ and in $C_{loc}^{2,1}((0,R^{n}]\times (0, \infty))$ as $\varepsilon\searrow 0$, and that it solves (2.19).
	\end{proof}
	
	\subsection{Classical solution of the hybrid problem (2.14)}	
	Since the existence of a solution $w\in C^{0}([0,R^{n}]\times [0, T))\cap C^{2,1}((0,R^{n}]\times (0, T))$ to the mixed problem (2.14) is obtained by studying the properties in a very short time. We can establish that
	\hspace*{\fill} \\
	\begin{lamme}
		Let $n\ge 2$, $R>0$, $\theta \in (0,1)$, $\mu =\frac{nm}{\omega_{n}R^{n}}$, $\beta >0$, $\alpha \ge 1$ and suppose that $w_{0}\in C^{1+\theta}([0,R^{n}])$ is as in $(2.18)$.\\
		Then there is a function $y(t)\in C^{1}([0,T^{*}))$ denote the solution to the following ODE system
		\begin{equation}
			\left\{    
			\begin{array}{ll}
				y'(t)=n^{\alpha}y^{\alpha+1}(t), & 0<t<T^{*},\\
				y(0)=y_{0},
			\end{array}
			\right.
		\end{equation}
		with
		\begin{equation}
			y_{0}:=\Arrowvert w_{0s} \Arrowvert_{L^{\infty}((0,R^{n}))}, 
		\end{equation}
		where $T^{*}>0$ denotes maximal time of existence of $(2.20)$. Furthermore, $w$ as in Lemma 2.3 satisfies
		\begin{equation}
			w(s,t)\le y(t)\cdot s\qquad for \ all \ (s,t)\in (0,R^{n}]\times[0,T^{*})
		\end{equation}
		Consequently,
		\begin{equation}
			w\in C^{0}([0,R^{n}]\times[0,T^{*})) \ with \ w(0,t)=0 \quad for \ all \ t\in [0,T^{*}).
		\end{equation}
	\end{lamme} 
	\begin{proof}
		Since $0\le w \le \frac{m}{\omega_{n}}$, $y(t)\ge 0$. For $\varepsilon\in  (0,R^{n})$ and arbitrary $T\in (0,T^{*})$, set
		\begin{equation}
			\overline{w}(s,t):= y(t)\cdot s, \qquad (s,t)\in [\varepsilon, R^{n}]\times[0,T].
		\end{equation}
		Then we have 
		\begin{equation}
			\begin{array}{l}
				\overline{w}_{t}-n^{2}s^{\frac{2n-2+\beta}{n}}\overline{w}_{ss}-n^{\alpha}\overline{w}\overline{w}_{s}^{\alpha}+n^{n-1}\mu (s-\varepsilon)\overline{w}^{\alpha}_{s}\\ \qquad\qquad=y'(t)s-n^{\alpha}y^{\alpha+1}(t)s+n^{\alpha-1}\mu (s-\varepsilon)y^{\alpha}(t)\\
				\qquad\qquad=n^{\alpha-1}\mu (s-\varepsilon)y^{\alpha}(t)\ge 0 
			\end{array}
		\end{equation}
		in $(\varepsilon, R^{n})\times(0,T)$. Furthermore, by $(2.21)$ we deduce that for $s\in (\varepsilon,R^{n})$
		\begin{equation}
			\begin{array}{ll}
				w_{\varepsilon}(s, 0) & =w_{0\varepsilon}(s)\le w_{0}(s)=w_{0}(0)+\int_{0}^{s}w_{0s}(\rho)d\rho\\
				\qquad & \le\Arrowvert w_{0s} \Arrowvert_{L^{\infty}((0,R^{n}))}\cdot s\\
				\qquad & =y_{0}\cdot s=\overline{w}(s, 0).
			\end{array}
		\end{equation}
		From $(2.26)$ and $w_{0\varepsilon s}\ge 0$ in $[\varepsilon, R^{n}]$, we know that $y_{0}\ge \frac{m}{R^{n}\omega_{n}}=\frac{\mu }{n}$. Moreover, due to $y'(t)=n^{\alpha}y^{\alpha+1}(t)\ge 0$ and $y(0)\ge 0$ , we can obtain that for all $t\ge 0$ we have $y(t)\ge \frac{\mu }{n}$. Therefore, 
		\begin{equation}
			w_{\varepsilon}(\varepsilon, t)=0\le y(t)\cdot \varepsilon =\overline{w}(\varepsilon, t)
		\end{equation}
		and 
		\begin{equation}
			w_{\varepsilon}(R^{n}, t)=\frac{m}{\omega_{n}}=\frac{\mu }{n}R^{n}\le y(t)\cdot R^{n} =\overline{w}(R^{n}, t) \quad for \ all \ t\in (0,T).
		\end{equation}
		Hence, By $(2.16), \ (2.25)-(2.28)$, and by Lemma 5.1, 
		$$
			w_{\varepsilon}(s,t)\le \overline{w}(s,t)=y(t)\cdot s \qquad for \ all \ (s,t)\in [\varepsilon, R^{n}]\times [0,T].
		$$
		Obviously, taking $T \nearrow T^{*}$ this inequality holds in $(s,t)\in [\varepsilon, R^{n}]\times [0,T^{*})$. Finally, because $\varepsilon \in (0, R^{n})$ has been chose arbitrarily, let $\varepsilon \searrow 0$, we can obtain $(2.22)$ and in addition $(2.23)$ is obvious. 
	\end{proof}
	Notice that Lemma 2.3 and Lemma 2.4 construct a solution to the system (2.14) belonging to $[0,R^{n}]\times[0,T^{*})$. Below we make a uniqueness claim about the solution of (2.14).
	\hspace*{\fill} \\
	\begin{lamme}
		Let $n\ge 2$, $R>0$, $\theta \in (0,1)$, $\mu =\frac{nm}{\omega_{n}R^{n}}$, $\beta >0$, $\alpha \ge 1$ and assume that $w_{0}\in C^{1+\theta}([0,R^{n}])$ is as in $(2.18)$.\\
		If the function $w$ is as defined in Lemma 2.4, let $T>0$ be such that
		\begin{equation}
			w\in C^{0}([0,R^{n}]\times[0,T)) \ with \ w(0,t)=0 \quad for \ all \ t\in (0,T),
		\end{equation}
		and  
		$$w_{s}\in L^{\infty}((0,R^{n})\times (0,T)).$$
		Then there exists a unique solution $w\in C^{0}([0,R^{n}]\times [0, T))\cap C^{2,1}((0,R^{n}]\times (0, T))$ to problem $(2.14)$.
	\end{lamme}
	\begin{proof}
		By $(2.19)$ and $(2.29)$, $w$ is a classical solution of $(2.14)$ in $[0, R^{n}] \times [0, T)$. Assume that $w_{1}$ is also a solution of $(2.14)$, and since $w_{s}\in L^{\infty}((0,R^{n})\times (0,T))$, use Lemma 5.1 to obtain
		$$
			w(s,t)=w_{1}(s,t) \qquad for \ all \ (s,t)\in [0,R^{n}]\times [0,T). 
		$$
		This proof is complete.
	\end{proof}
	From this, the study of boundedness in $w_{s}$ can be found to be necessary. However, the solution of $(2.14)$ that we constructed in Lemma 2.4 is not sufficient to derive $w_{s}$ is bounded. Thus, here we will impose better restrictions on the initial data to ensure that $w$ is a convex function with respect to $s$.
	\hspace*{\fill} \\
	\begin{lamme}
		Let $n\ge 2$, $R>0$, $\theta \in (0,1)$, $\mu =\frac{nm}{\omega_{n}R^{n}}$, $\beta >0$, $\alpha \ge 1$ and $w_{0}\in C^{2+\theta}([0,R^{n}])$ be as in $(2.18)$ with
		\begin{equation}
			w_{0ss}(s)\le 0, \quad s\in (0,R^{n})
		\end{equation}
		as well as 
		\begin{equation}
			w_{0ss}(0)=0, \ w_{0s}(R^{n})=0 \ and \ w_{0ss}(R^{n})=0. 
		\end{equation}
		Then, $w\in C^{2,1}((0,R^{n}]\times [0, \infty))$ as in Lemma 2.3 with
		\begin{equation}
			w_{ss}(s,t)\le 0 \qquad for \ all \ (s,t)\in (0, R^{n})\times [0,\infty).
		\end{equation}
	\end{lamme}
	\begin{proof}
		Defined $w_{0\varepsilon}\in C^{1+\theta}([\varepsilon, R^{n}])$ as $(2.17)$.
		From $(2.31)$ we may infer that
		$$
			w_{0\varepsilon ss}(\varepsilon)=0, \ w_{0\varepsilon ss}(R^{n})=0, \ w_{0\varepsilon s}(R^{n})=0.
		$$
		Let $w_{\varepsilon}\in C^{0}([\varepsilon,R^{n}]\times [0, \infty))\cap C^{2,1}([\varepsilon,R^{n}]\times (0, \infty))$, where $\varepsilon\in (0,R^{n})$ is arbitrarily.\\
		Hence 
		$$
			n^{2}\varepsilon^{\frac{2n-2+\beta}{n}}w_{0\varepsilon ss}(\varepsilon)+n^{\alpha}w_{0\varepsilon}(\varepsilon)w_{0\varepsilon s}^{\alpha}(\varepsilon)-n^{\alpha -1}\mu (\varepsilon-\varepsilon)w_{0\varepsilon s}^{\alpha}(\varepsilon)=0
		$$
		and
		$$
		n^{2}R^{n(\frac{2n-2+\beta}{n})}w_{0\varepsilon ss}(R^{n})+n^{\alpha}w_{0\varepsilon}(R^{n})w_{0\varepsilon s}^{\alpha}(R^{n})-n^{\alpha -1}\mu (R^{n}-\varepsilon)w_{0\varepsilon s}^{\alpha}(R^{n})=0
		$$
		Hence, the compatibility conditions are satisfied, allowing $w_{\varepsilon}\in C^{2,1}([\varepsilon,R^{n}]\times [0, \infty))$ to be derived by Schauder estimation. Let $\varepsilon \searrow 0$ then we get the desired regularity with respect to $w$. \\
		Since $w_{0\varepsilon ss}\le 0$, a standard comparison argument yields
		\begin{equation}
			w_{\varepsilon ss}(s, t)\le 0 \qquad for \ all \ (s,t)\in (\varepsilon, R^{n})\times [0, \infty), 
		\end{equation}
		whence we can immediately infer (2.32).
	\end{proof}
	\hspace*{\fill} \\
	\begin{lamme}
		Let $n\ge 2$, $R>0$, $\theta \in (0,1)$, $\mu =\frac{nm}{\omega_{n}R^{n}}$, $\beta >0$, $\alpha \ge 1$ and $w_{0}\in C^{2+\theta}([0,R^{n}])$ as described in $(2.18)$ with $(2.30)$ and $(2.31)$. \\
		Let $w$ denote the global solution of (2.16) in Lemma 2.3. Then, as shown in Lemma 2.4, when $T^{*}>0$, we have that for each $T\in (0,T^{*})$ there exists $C=C(T)>0$ such that
		\begin{equation}
			w_{s}(s,t)\le C \qquad for \ all \ (s,t)\in (0, R^{n}]\times [0,T].
		\end{equation}
	\end{lamme}
	\begin{proof}
		Let $0<T<T^{*}$, Lemma 2.4 ensures that
		$$
			w(s,t)\le y(T)\cdot s \le Cs \qquad (s,t)\in (0,R^{n}]\times [0,T]
		$$
		for some $C=C(T)>0$. Thus necessarily for $t\in (0,T)$
		$$
			\liminf_{s\searrow 0}w_{s}(s,t)\le C. 
		$$
		In addition, by Lemma 2.6, $w\in C^{2,1}((0,R^{n}]\times [0, \infty))$ satisfies
		\begin{equation}
			w_{ss}(s,t)\le 0 \qquad for \ all \ (s,t)\in (0, R^{n})\times [0,\infty).
		\end{equation}
		We can obtain $w_{s}(\cdot, t)$ to be monotonically decreasing for all $t\in [0,\infty)$, then $(2.34)$.
	\end{proof}

	\section{Existence uniqueness of solutions to the Keller-Segel system}
	Now using the local existence of a solution to the hybrid problem $(2.14)$, we accordingly obtain a local solution of the radial problem $(2.1)$ in radial coordinates, or more precisely, a local solution of the Keller-Segel problem $(1.1)$.
	\hspace*{\fill} \\
	\begin{theorem}
		Suppose that $n\ge 2$, $R>0$, $\Omega =B_{R}(0)\subset \mathbb{R}^{n}$, $\beta>0$, $\alpha \ge 1$, and write $\Omega_{0}:=\overline{\Omega} \setminus \{0\}$. For $\theta\in (0,1)$ and $u_{0}\in C^{\theta}(\overline{\Omega})$ conform to $(1.2)$, then there exists a radially symmetric classical solution $(u,v)$ of $(1.1)$ satisfying Definition 1.1, fulfilling
		\begin{equation}
			\left\{    
			\begin{array}{lll}
				u\in C^{0}(\Omega_{0}\times [0,T_{0}))\cap C^{2,1}(\Omega_{0}\times (0,T_{0})),\\
				v\in C^{2,0}(\Omega_{0}\times (0,T_{0})),
			\end{array}
			\right.
		\end{equation}
		for some $T_{0}\in (0,\infty]$. This solution $u$ is nonnegative and satisfies the conservation of mass property, i.e.
		\begin{equation}
			\int_{\Omega}u(\cdot, t)=\int_{\Omega}u_{0}=:m\qquad for \ all \ t\in (0,T_{0}). 
		\end{equation}
	\end{theorem}
	\hspace*{\fill} \\
	\begin{theorem}
		Assume that the conditions of Theorem 1.1 are satisfied and that $(u, v)$ denotes the classical solution of (1.1) identified in Theorem 1.1.\\
		Suppose that at some $\theta\in (0,1)$, the other $u_{0}\in C^{1+\theta}(\overline{\Omega})$ is radially decreasing and has the following properties
		\begin{equation}
			u_{0}=0 \ and \ \bigtriangledown u_{0}\cdot r=0 \qquad on \ \partial\Omega
		\end{equation}
		as well as 
		\begin{equation}
			|\bigtriangledown u_{0}(x)|\le c_{0}|x|^{n-1+\theta} \quad for \ some \ c_{0}>0.
		\end{equation}
		Then for $T_{0}>0$ as in Theorem 3.1 we have that
		\begin{equation}
			u\in C^{1,0}(\Omega_{0}\times [0,T_{0})) \quad is \ radially \ decreasing.
		\end{equation}
		Moreover, for some $T^{*}\in (0,T_{0}]$ and each $T\in (0, T^{*})$, there exists $C=C(T)>0$ such that 
		\begin{equation}
			u(x,t)\le C \qquad for \ all \ (x,t)\in \Omega_{0}\times [0,T].
		\end{equation}
		In conclusion, there exists a unique solution $(u, v)$ of $(1.1)$ in $\Omega_{0}\times [0,T^{*})$, i.e.
		\begin{equation}
			\left\{    
			\begin{array}{lll}
				u\in C^{0}(\Omega_{0}\times [0,T^{*}))\cap C^{2,1}(\Omega_{0}\times (0,T^{*})),\\
				v\in C^{2,0}(\Omega_{0}\times (0,T^{*})),
			\end{array}
			\right.
		\end{equation}
		which has the properties that
		\begin{equation}
			 0\le u\in L^{\infty}(\Omega\times(0,T)) \ and \ \bigtriangledown v \in L^{\infty}(\Omega\times(0,T); R^{n})
		\end{equation}
		for all $T\in (0,T^{*})$.
	\end{theorem}
	\hspace*{\fill} \\
	Notice that combining multiple results from the previous section, we are able to give a proposition about the existence of solutions of the form (2.1).
	\hspace*{\fill} \\
	\begin{lamme}
		Let $n\ge 2$, $R>0$, $\theta \in (0,1)$, $\mu =\frac{nm}{\omega_{n}R^{n}}$, $\beta >0$, $\alpha \ge 1$ and assume that $w_{0}\in C^{1+\theta}([0,R^{n}])$ is as in $(2.18)$.\\
		Moreover, we choose the maximum $T_{0} > 0$ such that
		$$
			w\in C_{loc}^{1,\frac{1}{2}}((0,R^{n}]\times [0, \infty))\cap C_{loc}^{2,1}((0,R^{n}]\times (0, \infty))
		$$
		is constructed by Lemma 2.3 is in $C^{0}([0,R^{n}]\times[0,T_{0}))$ with 
		$$
			w(0,t)=0 \quad for \ all \ t\in [0,T_{0}). 
		$$
		Then for $u_{0}\in C^{\theta}([0, R])$ defined via
		\begin{equation}
			u_{0}(r)=n\cdot w_{0s}(r^{n}), \qquad r\in [0,R], 
		\end{equation}
		the functions $u\in C^{0}((0, R]\times [0,T_{0}))\cap C^{2,1}((0, R]\times (0,T_{0}))$ defined by
		\begin{equation}
			u(r,t)=n\cdot w_{s}(r^{n}, t),
		\end{equation}
		and $v\in C^{2,0}((0, R]\times (0,T_{0}))$ with
		\begin{equation}
			v_{r}(r,t)=\frac{1}{r^{n-1}} \Big (\frac{\mu r^{n}}{n}-\int_{0}^{r}\rho ^{n-1}u(\rho, t)d\rho \Big ) \qquad (r,t)\in (0,R]\times (0,T_{0}), 
		\end{equation}
		solves $(2.1)$ classically in $(0,R]\times [0,T_{0})$.\\
		Moreover, with $T^{*}$ as in Lemma 2.4,
		\begin{equation}
			T_{0}\le T^{*}.
		\end{equation}
		Furthermore, $u$ is nonnegative and satisfies that mass is conserved, i.e.
		\begin{equation}
			\int_{0}^{R}\rho ^{n-1}u(\rho, t)d\rho =\int_{0}^{R}\rho ^{n-1}u_{0}(\rho)d\rho
		\end{equation}
		for all $t\in (0,T_{0})$.
	\end{lamme}
	\begin{proof}
		Since $w \in C_{loc}^{2,1}((0,R^{n}]\times (0, \infty))$ and $(2.14)$, for any $t>0$ we necessarily have
		$$
			w_{t}(R^{n},t)=n^{2}R^{n(\frac{2n-2+\beta}{n})}w_{ss}(R^{n},t)+n^{\alpha}w(R^{n},t)w_{s}^{\alpha}(R^{n},t)-n^{\alpha-1}\mu R^{n}w_{s}^{\alpha}(R^{n},t).
		$$
		Since $w(R^{n},t)=\frac{m}{\omega_{n}}=\frac{\mu R^{n}}{n}$ for all $t>0$ and thus also $w_{t}(R^{n},t)=0$ in $(0, \infty)$, this yields
		$$
			0=n^{2}R^{n(\frac{2n-2+\beta}{n})}w_{ss}(R^{n},t)+w_{s}^{\alpha}(R^{n},t) \Big (n^{\alpha}\frac{\mu R^{n}}{n}-n^{\alpha-1}\mu R^{n} \Big)
		$$
		and therefore, $w_{ss}(R^{n},t)=0\quad for \ all \ t>0$. 
		By $(3.10)$, we can obtain $u_{r}(r,t)=n^{2}w_{ss}(r^{n},t)r^{n-1}$. Therefore, $u_{r}(R,t)=0$ for all $t\in (0, T_{0})$.
	\end{proof}
	In the following we can prove Theorem 3.1.
	\begin{proof}[Proof of the Theorem 3.1]
		Let $m:=\int_{\Omega}u_{0}$. Assume that $\widetilde{u}_{0}\in C^{\theta}([0,R])$ is the value corresponding to $u_{0}$ after radial symmetry. Then $w_{0}: [0,R^{n}]\to R$ defined by
		$$
			w_{0}(s)=\int_{0}^{s^{\frac{1}{n}}}\rho^{n-1}\widetilde{u}_{0}d\rho \qquad for \ all \ s\in [0, R^{n}]
		$$
		is in $C^{1+\theta}([0,R^{n}])$ and satisfies
		$$
			n\cdot w_{0s}(r^{n})=\widetilde{u}_{0}(r) \qquad for \ all \ r\in [0,R]
		$$
		as well as $(2.18)$, and thus with $\mu :=\frac{nm}{\omega_{n}R^{n}}$, Lemma 2.3 and Lemma 2.4 assert the existence of a function $w\in C^{0}([0,R^{n}]\times [0, T))\cap C^{2,1}((0,R^{n}]\times (0, T))$ that is a solution of the problem $(2.14)$. Finally, it follows from Lemma 3.1 that $(2.1)$ has a solution $(\widetilde{u},\widetilde{v})$ in $(0,R]\times[0,T_{0})$ satisfying
		$$
		 	\left\{    
		 	\begin{array}{lll}
		 		\widetilde{u}\in C^{0}((0, R]\times [0,T_{0}))\cap C^{2,1}((0, R]\times (0,T_{0})),\\
		 		\widetilde{v}\in C^{2,0}((0, R]\times (0,T_{0})),
		 	\end{array}
		 	\right.
		$$
		and that $\widetilde{u}$ is nonnegative and satisfies conservation of mass.\\
		This implies that there exists a solution $(u,v)$ to $(1.1)$ in $\Omega_{0}\times [0,T_{0})$ satisfying
		$$
			\left\{    
			\begin{array}{lll}
				u\in C^{0}(\Omega_{0}\times [0,T_{0}))\cap C^{2,1}(\Omega_{0}\times (0,T_{0})),\\
				v\in C^{2,0}(\Omega_{0}\times (0,T_{0})),
			\end{array}
			\right. 
		$$
		and that $u$ is nonnegative and satisfies conservation of mass.
	\end{proof}
	The hybrid problem (2.14) has solution uniqueness, so it becomes reasonable for us to consider solution uniqueness of equation (2.1).
	\hspace*{\fill} \\
	\begin{lamme}
		Let $n\ge 2$, $R>0$, $\theta \in (0,1)$, $\mu =\frac{nm}{\omega_{n}R^{n}}$, $\beta >0$, $\alpha \ge 1$ and $u_{0}\in C^{\theta}([0,R])$.\\
		Then for $0<T<T^{*}$ there is at most one solution $(u, v)$ of (2.1) in $(0,R]\times [0,T)$ with
		\begin{equation}
			\left\{    
			\begin{array}{lll}
				u\in C^{0}((0, R]\times [0,T))\cap C^{2,1}((0, R]\times (0,T)),\\
				v\in C^{2,0}((0, R]\times (0,T)).
			\end{array}
			\right.
		\end{equation}
		which has the following property
		$$
			\int_{0}^{R} v(r,t)r^{n-1}dr=0 \qquad for \ all \ t\in (0,T), 	
		$$
		and
		\begin{equation}
			0\le u\in L^{\infty}((0,R)\times (0,T)) \ and \ v_{r}\in L^{\infty}((0,R)\times (0,T)).
		\end{equation}
	\end{lamme}
	\begin{proof}
		In the above sense, Lemma 3.1 guarantees the existence of a solution $(u, v)$ to the radial problem $(2.1)$.\\
		Notice that if $(3.15)$ holds, satisfies the requirements of Lemma 2.1, so the existence of a solution to $(2.1)$ implies that there exists a solution $w\in C^{0}([0,R^{n}]\times [0, T))\cap C^{2,1}((0,R^{n}]\times (0, T))$ of $(2.14)$ with $w_{0}\in C^{1+\theta}([0,R^{n}])$, and $w_{s}\in C^{0}([0,R^{n}]\times [0,T))$ is nonnegative and bounded.\\
		Furthermore, Lemma 2.5 warrants the uniqueness of the solution to the hybrid problem $(2.14)$.\\
		In conclusion, the solutions of $(2.1)$ and $(2.14)$ are correspond, thus conveying their uniqueness. 
	\end{proof}
	Therefore, our further goal is to study the properties of the radialized $u$ and $v$ satisfying the equation $(3.15)$.
	\hspace*{\fill} \\
	\begin{lamme}
		Assuming that the conditions of Lemma 2.7 are satisfied and let $(u, v)$ indicate the solution of $(2.1)$ constructed in Lemma 3.1.\\
		Then, the solution has properties that $u\in C^{1,0}((0,R]\times [0,T_{0}))$ and
		\begin{equation}
			u_{r}(r,t)\le 0 \qquad for \ all \ (r,t)\in (0,R]\times [0, T_{0}).
		\end{equation}
		Moreover, for $T^{*}>0$ as in Lemma 2.4 and each $T\in (0,T^{*})$, there exists $C=C(T^{*})>0$ such that 
		\begin{equation}
			u(r,t)\le C \qquad for \ all \ (r,t)\in (0,R]\times[0,T].
		\end{equation}
	\end{lamme}
	\begin{proof}
		By Lemma 2.7 and $$w_{s}=\frac{1}{n}u(s^{\frac{1}{n}}, t)=\frac{1}{n}u(r,t), $$ $$w_{ss}=\frac{1}{n^{2}}u_{r}(s^{\frac{1}{n}}, t)s^{\frac{1-n}{n}}=\frac{1}{n^{2}}u_{r}(r, t)r^{1-n}, $$ we find that $(3.16)$ and $(3.17)$ are corollaries of $(2.32)$ and $(2.34)$.
	\end{proof}
	Thereby, Theorem 3.2 can be argued by the following method.
	\hspace*{\fill} \\
	\begin{proof}[Proof of the Theorem 3.2]
		Assume that $\widetilde{u}_{0}\in C^{1+\theta}([0,R])$ is the value corresponding to $u_{0}$ after radial symmetry. Then $w_{0}: [0,R^{n}]\to R$ defined by
		$$
		   w_{0}(s)=\int_{0}^{s^{\frac{1}{n}}}\rho^{n-1}\widetilde{u}_{0}d\rho \qquad for \ all \ s\in [0, R^{n}]
		$$
		is in $C^{2+\theta}((0,R^{n}])$ and satisfies
		$$
		   n\cdot w_{0s}(r^{n})=\widetilde{u}_{0}(r) \qquad for \ all \ r\in [0,R]
		$$
		as well as
		$$
			\widetilde{u}_{0r}=n^{2}w_{0ss}(r^{n})\cdot r^{n-1} \qquad for \ all \ r\in (0,R].
		$$
		Since $u_{0}$ is radially decreasing and thus $\widetilde{u}_{0r}\le 0 \ \text{in} \ (0,R]$, this implies
		\begin{equation}
			w_{0ss}(s)\le 0 \qquad for \ all \ s\in (0,R^{n}].
		\end{equation}
		Moreover, from $(3.3)$, we obtain
		\begin{equation}
			w_{0s}(R^{n}) = 0 \ and \ w_{0ss}(R^{n}) = 0\qquad on \ \partial \Omega. 
		\end{equation}
		Furthermore, because of $|\bigtriangledown u_{0}(x)|=|\widetilde{u}_{0r}(r)|$ with $r=|x|$ for $x\in \overline{\Omega}$, $(3.4)$ results in
		\begin{equation}
			|w_{0ss}(r^{n})|\le \frac{1}{n^{2}}r^{1-n}c_{0}r^{n-1+\theta}=\frac{c_{0}}{n^{2}}r^{\theta}\to 0
		\end{equation}
		for $r\to 0$, warranting that $w_{0}\in C^{2+\theta}([0,R^{n}])$ with
		\begin{equation}
			w_{0ss}(0)=0. 
		\end{equation}
		With (3.18), (3.19) and (3.21) we have that the conditions of Lemma 2.7 are satisfied, thus $\widetilde{u}\in C^{1,0}(0,R]\times [0,t_{0})$ with 
		\begin{equation}
			\widetilde{u}_{r}(r,t)\le 0 \qquad for \ all \ (r,t)\in (0,R]\times [0,T_{0}), 
		\end{equation}
		and there are also the fact that $T^{*}\in [0,T_{0})$ such that for $T\in (0,T^{*})$, we have
		\begin{equation}
			\widetilde{u}(r,t)\le C \qquad for \ all \ (r,t)\in (0,R]\times [0,T), 
		\end{equation}
		with some $C=C(T^{*})>0$ by Lemma 3.3. However, this implies that $(3.5)$ and $(3.6)$ are valid for the solution $(u, v)$ of $(1.1)$ in Theorem 3.1.\\		
		By the definition of $\widetilde{v}_{r}$ in (2.3) and (3.23), We can further deduce that
		$$
			\begin{array}{ll}
				|\widetilde{v}_{r}(r,t)| & =|\frac{1}{r^{n-1}}\Big (\frac{\mu r^{n}}{n}-\int_{0}^{r}\rho^{n-1}\widetilde{u}(\rho, t)d\rho \Big)|\\
				\qquad & \le \frac{1}{r^{n-1}}\cdot \frac{\mu r^{n}}{n}+C(T^{*})\frac{1}{r^{n-1}}\int_{0}^{r}\rho^{n-1}d\rho\\
				\qquad & =\frac{\mu r}{n}+C(T^{*})\frac{r}{n}
			\end{array}
		$$
		for $T\in (0,T^{*})$ and $(r,t)\in (0,R]\times (0,T)$, warranting
		\begin{equation}
			\widetilde{v}_{r}\in L^{\infty}((0,R)\times(0,T))\qquad for \ all \ T\in (0,T^{*}).
		\end{equation}
		Since $\int_{0}^{R} \widetilde{v}(r,t)r^{n-1}dr=0$, this allows us to uniquely confirm that $\widetilde{v}\in C^{2,0}((0,R]\times (0,T^{*}))$.\\
		Then, since (3.23), (3.24) and nonnegativity of $\widetilde{u}$, the uniqueness of the solution in $(0,R]\times [0,T^{*})$ to the system of equations $(2.1)$ is guaranteed by Lemma 3.2 and satisfies property $(3.15)$.\\
		With respect to the original variables and taking into account that
		$$
			|\bigtriangledown v(x,t)|=|\widetilde{v}_{r}(|x|,t)|\qquad for \ all \ (x,t)\in \overline{\Omega}\times (0,T^{*}), 
		$$
		this implies that the corresponding pair of functions $(u, v)$ is indeed the one that satisfies $(3.7)$ and $(3.8)$ for which $(1.1)$ has a unique solution in $\Omega_{0}\times [0,T^{*})$. 
	\end{proof}

	\section{Blow up}
	Since the classical solution we define does not require $u$ to have a continuous definition on the tight space, it is likely that for the domain $\Omega \in \mathbb{R}^{n}$, when $t \nearrow T$, there is $\Vert u(\cdot, t) \Vert_{L^{\infty}(\Omega)}\to \infty$, where $T<\infty$ denotes the maximal existence time corresponding to the occurrence of a Keller-Segel blow up, but $u\in C^{0}(\Omega \times [0,T_{0}))\cap C^{2,1}(\Omega \times (0,T_{0}))$ for some $T_{0}>T$. In addition, in some cases there may be no global bounded solution.
	\hspace*{\fill} \\
	\begin{theorem}
		Suppose that $n\ge 2$, $R>0$, $\Omega =B_{R}(0)\subset \mathbb{R}^{n}$, $\beta>0$, $\alpha \ge 1$, and $u_{0}$ conform to (1.2).\\
		Then for $m:=\int_{\Omega}u_{0}$ and each $m_{0}\in (0,m]$, there exists $r_{0}=r_{0}(m_{0},m,R,\beta, \alpha)>0$ such that if
		\begin{equation}
			\int_{B_{r_{0}}(0)}\ge m_{0}, 
		\end{equation}
		there is no global classical solution $(u, v)$ of $(1.1)$ fulfilling
		\begin{equation}
			\left\{    
			\begin{array}{lll}
				u\in C^{0}(\Omega_{0}\times [0,\infty))\cap C^{2,1}(\Omega_{0}\times (0,\infty)),\\
				v\in C^{2,0}(\Omega_{0}\times (0,\infty)),
			\end{array}
			\right.
		\end{equation}
		such that for each $T\in (0,\infty)$
		\begin{equation}
			0\le u\in L^{\infty}(\Omega\times(0,T)) \ and \ \bigtriangledown v \in L^{\infty}(\Omega\times(0,T); R^{n}).
		\end{equation}
	\end{theorem}
	
	\subsection{Proofs of the Theorem 4.1}
	To prove the theorem, we begin with the following key lemma, which is shown by a method similar to that in \cite{fluchter2024solutions, winkler2019unstable}.
	\hspace*{\fill} \\
	\begin{lamme}
		Suppose that $n\ge 2$, $\beta > 0, \ \alpha \ge 1$, $R>0$, $m_{0}>0$ and $m\ge m_{0}$. There exists $s_{0}=s_{0}(m_{0},m,R,\beta, \alpha)\in (0, R^{n})$ such that if $w_{0}\in C^{1} ([0,R^{n}])$, $\mu =\frac{nm}{\omega_{n}R^{n}}$, and
		\begin{equation}
			w_{0}(s_{0})\ge \frac{m_{0}}{\omega_{n}}, 
		\end{equation}
		then there is no global classical solution
		$$
			w\in C^{0}([0,R^{n}]\times [0, \infty))\cap C^{2,1}((0,R^{n}]\times (0, \infty))
		$$
		of $(2.14)$ with the property that for each $T>0$
		\begin{equation}
			w_{s}\in L^{\infty}((0,R^{n})\times (0,T)).
		\end{equation}
	\end{lamme}
	\begin{proof}
		Since $\beta>0$ and $n\ge 2$, it is possible to fix $\gamma\in (0,1)$ with the property that $\beta>2-n$, hence
		$$
			\gamma\le 1-\frac{2}{n}+\frac{\beta}{n}.
		$$
		We assume that $w\in C^{0}([0,R^{n}]\times [0, \infty))\cap C^{2,1}((0,R^{n}]\times (0, \infty))$ is a global classical solution of (2.14) satisfying (4.5).\\
		For $s_{1}\in (0,R^{n})$ and $\delta \in (0, \frac{s_{1}}{2})$, we then use (2.14) to
		compute		
		\begin{equation}
			\begin{array}{lll}
				\frac{d}{dt} \int_{\delta}^{s_{1}}s^{-\gamma}(s_{1}-s)w(s,t)ds & = \int_{\delta}^{s_{1}}s^{-\gamma}(s_{1}-s)w_{t}(s,t)ds\\
				\qquad & = n^{2}\int_{\delta}^{s_{1}}s^{2-\frac{2}{n}+\frac{\beta}{n}-\gamma}(s_{1}-s)w_{ss}(s,t)ds\\ 
				\qquad & \quad + n^{\alpha}\int_{\delta}^{s_{1}}s^{-\gamma}(s_{1}-s)w(s,t)w_{s}^{\alpha}(s,t)ds\\
				\qquad & \quad - n^{\alpha-1}\mu \int_{\delta}^{s_{1}}s^{1-\gamma}(s_{1}-s)w_{s}^{\alpha}(s,t)ds\\
				\qquad &=I+II-III.
			\end{array}
		\end{equation}
		The first term on the right-hand side of the equation, 
		\begin{equation}
			\begin{array}{ll}
				I= & - n^{2}\delta^{2-\frac{2}{n}+\frac{\beta}{n}-\gamma}(s_{1}-\delta)w_{s}(\delta,t)\\
				\quad & -2n^{2}(2-\frac{2}{n}+\frac{\beta}{n}-\gamma)\int_{\delta}^{s_{1}}s^{1-\frac{2}{n}+\frac{\beta}{n}-\gamma}w(s,t)ds\\
				\quad & +n^{2}(2-\frac{2}{n}+\frac{\beta}{n}-\gamma)(1-\frac{2}{n}+\frac{\beta}{n}-\gamma)\int_{\delta}^{s_{1}}s^{-\frac{2}{n}+\frac{\beta}{n}-\gamma}(s_{1}-s)w(s,t)ds\\
				\quad & +n^{2}(2-\frac{2}{n}+\frac{\beta}{n}-\gamma)\delta^{1-\frac{2}{n}+\frac{\beta}{n}-\gamma}(s_{1}-\delta)w(\delta,t)\\
				\quad & +n^{2}s_{1}^{2-\frac{2}{n}+\frac{\beta}{n}-\gamma}w(s,t)-n^{2}\delta^{2-\frac{2}{n}+\frac{\beta}{n}-\gamma}w(\delta,t), 
			\end{array}
		\end{equation}
		is obtained directly by divisional integration.\\
		By means of Bernoulli's inequality deflation and integration by parts, II is computed to satisfy
		\begin{equation}
			\begin{array}{ll}
				II= & n^{\alpha}\int_{\delta}^{s_{1}}s^{-\gamma}(s_{1}-s)w(s,t)\Big[1+(w_{s}(s,t)-1)\Big]^{\alpha}ds\\
				\quad & \ge (1-\alpha)n^{\alpha}\int_{\delta}^{s_{1}}s^{-\gamma}(s_{1}-s)w(s,t)ds-\frac{\alpha n^{\alpha}}{2}\delta^{-\gamma}(s_{1}-\delta)w^{2}(\delta,t)\\
				\quad & +\frac{\alpha n^{\alpha}r}{2}\int_{\delta}^{s_{1}}s^{-\gamma-1}(s_{1}-s)w^{2}(s,t)ds+\frac{\alpha n^{\alpha}}{2}\int_{\delta}^{s_{1}}s^{-\gamma}w^{2}(s,t)ds.
			\end{array}
		\end{equation}
		By Holder's inequality and partial integration one can compute
		\begin{equation}
			\begin{array}{ll}
				III\le & -n^{\alpha-1}\mu C_{4}\delta^{1-\gamma}(s_{1}-\delta)w(\delta,t)\\
				\quad & -(1-\gamma)n^{\alpha-1}\mu C_{4}\int_{\delta}^{s_{1}} s^{-\gamma}(s_{1}-s)w(s,t)ds\\
				\quad & +n^{\alpha-1}\mu C_{4}\int_{\delta}^{s_{1}}s^{1-\gamma}w(s,t)ds, 
			\end{array}
		\end{equation}
		where $C_{4}:=\Vert w_{s} \Vert^{\alpha-1}_{L^{\infty}((0,R^n)\times (0,T))}$. \\
		Now, As we assumed that $w(0, t) = 0$ for all $t \ge 0$ and $w_{s}\in L^{\infty}((0,R^{n})\times(0,t))$, this implicitly implies that $sup_{(s,\tau)\in (0,R^{n})\times(0,t)}\frac{w(s,\tau)}{s}$ is finite, and thus
		\begin{equation}
			\begin{array}{l}
				n^{2}(2-\frac{2}{n}+\frac{\beta}{n}-\gamma)\delta^{1-\frac{2}{n}+\frac{\beta}{n}-\gamma}(s_{1}-\delta)w(\delta,t)\to 0, \\
				n^{2}\delta^{2-\frac{2}{n}+\frac{\beta}{n}-\gamma}w(\delta,t)\to 0, \\
				n^{2}\delta^{2-\frac{2}{n}+\frac{\beta}{n}-\gamma}(s_{1}-\delta)w_{s}(\delta,t)\to 0, \\
				\frac{\alpha n^{\alpha}}{2}\delta^{-\gamma}(s_{1}-\delta)w^{2}(\delta,t)\to 0, \\
				n^{\alpha-1}\mu C_{4}\delta^{1-\gamma}(s_{1}-\delta)w(\delta,t) \to 0 \qquad as \ \delta \searrow 0.
			\end{array}
		\end{equation}
		By $(4.6)-(4.10)$, letting $\delta \searrow 0$ and neglecting certain positive terms, and integrating in time shows that
		\begin{equation}
			\begin{array}{ll}
				\int_{0}^{s_{1}}s^{-\gamma}(s_{1}-s)w(s,t)ds & \ge \int_{0}^{s_{1}}s^{-\gamma}(s_{1}-s)w_{0}(s)ds\\
				\qquad & \quad -2n^{2}(2-\frac{2}{n}+\frac{\beta}{n}-\gamma)\int_{0}^{t}\int_{0}^{s_{1}}s^{1-\frac{2}{n}+\frac{\beta}{n}-\gamma}w(s,\tau)dsd\tau \\
				\qquad & \quad +\frac{\alpha n^{\alpha}}{2}\int_{0}^{t}\int_{0}^{s_{1}}s^{-\gamma}w^{2}(s,\tau)dsd\tau\\
				\qquad & \quad -n^{\alpha-1}\mu C_{4}\int_{0}^{t}\int_{0}^{s_{1}}s^{1-\gamma}w(s,\tau)dsd\tau\\
				\qquad & \quad +(1-\alpha)n^{\alpha}\int_{0}^{t}\int_{0}^{s_{1}}s^{-\gamma}(s_{1}-s)w(s,\tau)dsd\tau
			\end{array}
		\end{equation}
		By young inequality, 
		\begin{equation}
			\begin{array}{ll}
				2n^{2}(2-\frac{2}{n}+\frac{\beta}{n}-\gamma)\int_{0}^{s_{1}}s^{1-\frac{2}{n}+\frac{\beta}{n}-\gamma}w(s,\tau)ds\\
				\le \frac{\alpha n^{\alpha}}{8}\int_{0}^{s_{1}}s^{-\gamma}w^{2}(s,\tau)ds+\frac{8}{\alpha}(2-\frac{2}{n}+\frac{\beta}{n}-\gamma)^{2}n^{4-\alpha}\int_{0}^{s_{1}}s^{2-\frac{4}{n}+\frac{2\beta}{n}-\gamma}ds\\
				=\frac{\alpha n^{\alpha}}{8}\int_{0}^{s_{1}}s^{-\gamma}w^{2}(s,\tau)ds+c_{1}s_{1}^{3-\frac{4}{n}+\frac{2\beta}{n}-\gamma}
			\end{array}
		\end{equation}
		for all $\tau>0$ with $c_{1}:=\frac{8(2-\frac{2}{n}+\frac{\beta}{n}-\gamma)^{2}n^{4-\alpha}}{(3-\frac{4}{n}+\frac{2\beta}{n}-\gamma)\alpha}$, \\
		and
		\begin{equation}
			\begin{array}{ll}
				n^{\alpha-1}\mu C_{4}\int_{0}^{s_{1}}s^{1-\gamma}w(s,\tau)ds\\
				\le \frac{\alpha n^{\alpha}}{8}\int_{0}^{s_{1}}s^{-\gamma}w^{2}(s,\tau)ds+\frac{2}{\alpha}\mu^{2}C_{4}^{2}n^{\alpha-2}\int_{0}^{s_{1}}s^{2-\gamma}ds\\
				=\frac{\alpha n^{\alpha}}{8}\int_{0}^{s_{1}}s^{-\gamma}w^{2}(s,\tau)ds+c_{2}\frac{m^{2}}{R^{2n}}s_{1}^{3-\gamma}
			\end{array}
		\end{equation}
		for all $\tau>0$ with
		$c_{2}:=\frac{2C_{4}^{2}n^{\alpha}}{\alpha (3-\gamma)\omega_{n}^{2}}$. \\
		Define $y(t):=\int_{0}^{s_{1}}s^{-\gamma}(s_{1}-s)w(s,t)ds$ with $t\ge 0$, we can obtain
		\begin{equation}
			\begin{array}{ll}
				y(t) & \le \Big(\int_{0}^{s_{1}}s^{-\gamma}w^{2}(s,t)ds\Big)^{\frac{1}{2}}\Big(\int_{0}^{s_{1}}s^{-\gamma}(s_{1}-s)^{2}ds\Big)^{\frac{1}{2}}\\
				\quad & \le \Big(\int_{0}^{s_{1}}s^{-\gamma}w^{2}(s,t)ds\Big)^{\frac{1}{2}}\Big(\int_{0}^{s_{1}}s^{-\gamma}s_{1}^{2}ds\Big)^{\frac{1}{2}}\\
				\quad & =\Big(\frac{1}{1-\gamma}s_{1}^{3-\gamma} \Big)^{\frac{1}{2}}\Big(\int_{0}^{s_{1}}s^{-\gamma}w^{2}(s,t)ds\Big)^{\frac{1}{2}}
			\end{array}
		\end{equation}
		for all $t>0$ by the H$\ddot{o}$lder's inequality. 
		Thus, $(4.11)-(4.14)$ leads to
		\begin{equation}
			\begin{array}{ll}
				y(t)\ge & y(0)+\frac{\alpha n
				^{\alpha}(1-\gamma)}{4}s_{1}^{\gamma-3}\int_{0}^{t}y^{2}(\tau)d\tau +(1-\alpha)n^{\alpha}\int_{0}^{t}y(\tau)d\tau\\
				\qquad & -\Big(c_{1}s_{1}^{3-\frac{4}{n}+\frac{2\beta}{n}-\gamma}+c_{2}\frac{m^{2}}{R^{2n}}s_{1}^{3-\gamma}\Big)t. 
			\end{array}
		\end{equation}
		Observe that $\beta >2-n$, which implies that $2-\frac{4}{n}+\frac{2\beta}{n}\ge 0$.\\
		Thus, let $R>0$, $m_{0}>0$ and $m>0$ we can fix $s_{0}=s_{0}(m_{0}, m, R, \alpha , \beta)\in (0, \frac{R^{n}}{2})$ such that $s_{1}:=2s_{0}$ fulfills
		\begin{equation}
			s_{1}^{2-\frac{4}{n}+\frac{2\beta}{n}}\le \frac{\alpha n^{\alpha}(1-\gamma)c_{3}^{2}}{12c_{1}}m_{0}^{2},
		\end{equation}
		\begin{equation}
			s_{1}^{2}\le \left\{    
				\begin{array}{lll}
					\min \Biggl\{\frac{\alpha n^{\alpha}(1-\gamma)c_{3}^{2}}{12c_{2}}\frac{m_{0}^{2}R^{2n}}{m^{2}}, \frac{\alpha^{2}(1-\gamma)^{2}c_{3}^{2}}{144(\alpha-1)^{2}}m_{0}^{2}\Biggr\}, & \alpha >1,\\
					\frac{\alpha n^{\alpha}(1-\gamma)c_{3}^{2}}{6c_{2}}\frac{m_{0}^{2}R^{2n}}{m^{2}}, & \alpha =1,
				\end{array}
			\right.
		\end{equation}
		with $c_{3}=\frac{1}{\omega_{n}}\Big[\frac{1}{1-\gamma}(1-\frac{1}{2^{1-\gamma}})-\frac{1}{2-\gamma}(1-\frac{1}{2^{2-\gamma}})\Big]$.
		Here since (4.4) along with our selections of $s_{0}$  ensures that
		\begin{equation}
			\begin{array}{ll}
				y(0) & \ge \frac{m_{0}}{\omega_{n}}\int_{\frac{s_{1}}{2}}^{s_{1}}s^{-\gamma}(s_{1}-s)ds\\
				\quad & =\frac{m_{0}}{\omega_{n}}\Big[\frac{1}{1-\gamma}s_{1}(s_{1}^{1-\gamma}-(\frac{s_{1}}{2})^{1-\gamma})-\frac{1}{2-\gamma}(s_{1}^{2-\gamma}-(\frac{s_{1}}{2})^{2-\gamma})\Big]\\
				\quad & =\frac{m_{0}}{\omega_{n}}\Big[\frac{1}{1-\gamma}(1-\frac{1}{2^{1-\gamma}})-\frac{1}{2-\gamma}(1-\frac{1}{2^{2-\gamma}})\Big]s_{1}^{2-\gamma}\\
				\quad & =c_{3}m_{0}s_{1}^{2-\gamma}.
			\end{array}
		\end{equation}
		Therefore, by (4.16) and (4.17),
		$$
			\begin{array}{l}
				\frac{c_{1}s_{1}^{3-\frac{4}{n}+\frac{2\beta}{n}-\gamma}+c_{2}\frac{m^{2}}{R^{2n}}s_{1}^{3-\gamma}+(\alpha-1)n^{\alpha}y(0)}{\frac{\alpha n^{\alpha} (1-\gamma)}{4}s_{1}^{\gamma-3}y^{2}(0)}\\
				\qquad\le \frac{4c_{1}}{\alpha n^{\alpha} (1-\gamma)c_{3}^{2}m_{0}^{2}}s_{1}^{2-\frac{4}{n}+\frac{2\beta}{n}}+
				\frac{4c_{2}m^{2}}{\alpha n^{\alpha} (1-\gamma)R^{2n}c_{3}^{2}m_{0}^{2}}s_{1}^{2}\\
				\qquad\quad +\frac{4(\alpha-1)n^{\alpha}}{\alpha n^{\alpha} (1-\gamma)c_{3}m_{0}}s_{1}\\
				\qquad\le 1, 
			\end{array}
		$$
		Obviously, there exists $T > 0$ such that the problem
		\begin{equation}
			\left\{    
			\begin{array}{l}
				\underline{y}'(t)=\frac{\alpha n^{\alpha} (1-\gamma)}{4}s_{1}^{\gamma-3}\underline{y}^{2}(t)-(\alpha-1)n^{\alpha}\underline{y}(t)\\
				\qquad\qquad\qquad\qquad\qquad\qquad-\Big(c_{1}s_{1}^{3-\frac{4}{n}+\frac{2\beta}{n}-\gamma}+c_{2}\frac{m^{2}}{R^{2n}}s_{1}^{3-\gamma}\Big),\quad t\in (0, T),\\ 
				\underline{y}(0)=y(0),
			\end{array}
			\right.
		\end{equation}
		has a solution $\underline{y}\in C^{1}([0,T])$ with $\underline{y}(t)\nearrow +\infty$ as $t\nearrow T$. Finally, a standard comparison argument based on $(4.15)$ ensures that $y(t)\ge \underline{y}(t)$ for all $t\in (0,T)$, which is incompatible with our assumption.
	\end{proof}
	\begin{proof}[Proof of the Theorem 4.1]
		Let $n\ge 2$, $R>0$, $\Omega =B_{R}(0)\subset \mathbb{R}^{n}$, $\beta>0$, $\alpha \ge 1$ and $u_{0}\in C^{\theta}([0,R])$ conform to (1.2).\\
		Assume that $\widetilde{u}_{0}\in C^{\theta}([0,R])$ is the value corresponding to $u_{0}$ after radial symmetry. Moreover, suppose that $m_{0}:=\int_{\Omega} u_{0}$, $m_{0}\in (0,m]$ and $r_{0}=s_{0}^{\frac{1}{n}}$ for $s_{0}=s_{0}(m_{0},m,R,\beta, \alpha)\in (0, R^{n})$ as in Lemma 4.1 such that
		\begin{equation}
			\int_{B_{r_{0}}(0)} u_{0}\ge m_{0}
		\end{equation} 
		holds. Furthermore, suppose that there exists a global classical solution to (1.1) satisfying (4.2) and (4.3) for all $T > 0$. The following assumes that $(\widetilde{u}, \widetilde{v})$ as $(u, v)$ in radial coordinates and has the property
		\begin{equation}
			\left\{    
			\begin{array}{lll}
				\widetilde{u}\in C^{0}((0, R]\times [0,\infty))\cap C^{2,1}((0, R]\times (0,\infty)),\\
				\widetilde{v}\in C^{2,0}((0, R]\times (0,\infty)),
			\end{array}
			\right.
		\end{equation}
		and solves (2.1) for the initial condition $\widetilde{u}(\cdot, 0)=\widetilde{u}_{0}$.\\
		Since $|\bigtriangledown v(x,t)|=|\widetilde{v}_{r}(|x|,t)| \ for \ all \ (x,t)\in \Omega_{0}\times (0,\infty)$,
		\begin{equation}
			0\le \widetilde{u}\in L^{\infty}((0,R)\times(0,T)) \ and \ \widetilde{v}_{r}\in L^{\infty}((0,R)\times(0,T)).
		\end{equation}
		As $n\ge 2$, Lemma 2.1 guarantee that
		$$
		   \widetilde{v}_{r}(r,t)=\frac{1}{r^{n-1}}\Big (\frac{\mu r^{n}}{n}-\int_{0}^{r}\rho^{n-1}\widetilde{u}(\rho, t)d\rho \Big)\quad for \ all \ (s,t)\in (0,R]\times [0,T).
		$$
		Hence, $w: [0,R^{n}]\times [0,\infty)\to R$ defined by
		$$
		     w(s,t):=\int_{0}^{s^{\frac{1}{n}}}\rho^{n-1}\widetilde{u}(\rho, t)d\rho \qquad s=r^{n}\in [0, R^{n}], \ t\in [0, \infty),
		$$
		is in $w\in C^{0}([0,R^{n}]\times [0, T))\cap C^{2,1}((0,R^{n}]\times (0, T))$ and solves (2.14) in $[0,R^{n}]\times [0, \infty)$ for $w_{0}\in C^{1}([0,R^{n}])$ given by
		$$
		    w_{0}(s)=\int_{0}^{s^{\frac{1}{n}}}\rho^{n-1}\widetilde{u}_{0}(\rho)d\rho \qquad for \ all \ s\in [0, R^{n}].
		$$
		Moreover,
		$$
			w_{s}(s,t)=\frac{1}{n}\widetilde{u}(s^{\frac{1}{n}},t) \qquad for \ all \ (s,t)\in (0,R^{n})\times(0,\infty)
		$$
		and (4.22) entails that for each $T > 0$
		$$
			w_{s}\in L^{\infty}((0,R^{n})\times (0,T)).
		$$
		Since however for $s_{0}\in (0, R^{n})$ as above by (4.1),
		$$
			\begin{array}{ll}
				w_{0}(s_{0})&=\int_{0}^{s_{0}^{\frac{1}{n}}}\rho^{n-1}\widetilde{u}_{0}(\rho)d\rho\\
				\quad & =\frac{1}{\omega_{n}}\int_{B_{\frac{1}{n}}(0)}u_{0}\\
				\quad & =\frac{1}{\omega}\int_{B_{r_{0}}(0)}u_{0}\\
				\quad & \ge \frac{m_{0}}{\omega_{n}},
			\end{array}
		$$
		this is inconsistent with Lemma 4.1, thus $(1.1)$ does not have an global classical solution satisfying $(4.2)$ and $(4.3)$ under the given initial value assumption.
	\end{proof}

	\newpage
	\section{Appendix: A comparison principle of (2.14) and (2.16)}	
	\indent Since it is also necessary to be able to deal with diffusion degeneracy in the course of our arguments, the standard principle of comparison do not apply. In addition, the comparison principle constructed in this paper is more complex relative to the lack of certain conditions mentioned in the literature \cite{winkler2019unstable, bellomo2017finite} and relative to the literature \cite{fluchter2024solutions}. Therefore, we prove a variant of the parabolic comparison principle that precisely covers the present situation.
	\begin{lamme}
		Let $T>0$ and $l,L \ge 0$ with $l<L$. Assume that the functions $\overline{w}$ and $\underline{w}$ belong to $C^{0}([l,L]\times [0,T])\cap C^{2,1}((l,L)\times (0,T])$. In addition, either \\ $$\underline{w}_s \in L^{\infty} ((l,L)\times (0,T)), \ or \ \overline{w}_s \in L^{\infty} ((l,L)\times (0,T))$$. \\
		Furthermore, for $a, b, \gamma, \delta \ge 0$, $\alpha>0$, and $\theta, c, d \in \mathbb{R}$, if
		\begin{equation}
			\underline{w}_{t} \le as^{\theta}\underline{w}_{ss}+bs^{\gamma}\underline{w}\underline{w}_{s}^{\alpha}+cs^{\delta}\underline{w}_{s}^{\alpha}+d\underline{w}_{s}^{\alpha}
		\end{equation}
		and 
		\begin{equation}
			\overline{w}_{t} \ge as^{\theta}\overline{w}_{ss}+bs^{\gamma}\overline{w}\overline{w}_{s}^{\alpha}+cs^{\delta}\overline{w}_{s}^{\alpha}+d\overline{w}_{s}^{\alpha}
		\end{equation}
		are hold for all $(s,t)\in (l,L)\times (0,T)$ with
		\begin{equation}
			\underline{w}(s,0) \le \overline{w}(s,0) \qquad for \ all \ s \in (l,L)		
		\end{equation}
		and
		\begin{equation}
			\underline{w}(l,t) \le \overline{w}(l,t), \ \underline{w}(L,t) \le \overline{w}(L,t) \qquad for \ all \ t \in (0,T).
		\end{equation}
		Then
		\begin{equation}
			\underline{w}(s,t) \le \overline{w}(s,t) \qquad for \ all \ s\in [l,L] \ and \ t\in [0,T).
		\end{equation}
	\end{lamme}
	\begin{proof}
		Assuming $\overline{w}_s \in L^{\infty} ((l,L)\times (0,T))$, we let 
		$$
			bs^{r}|\overline{w}_s^{\alpha}|\le C_{1}\qquad for \ all \ (s,t)\in (l,L)\times (0,T)
		$$
		where $C_{1}$ is a constant. For $\varepsilon >0$, we define 
		$$
			z(s,t):=\underline{w}(s,t)-\overline{w}(s,t)-\varepsilon e^{2C_{1}t}\qquad for \ all \ (s,t)\in[l,L]\times[0,T]
		$$
		and claim that 
		\begin{equation}
			z(s,t)<0 \qquad for \ all \ (s,t)\in (l,L)\times (0,T).
		\end{equation}
		\indent To verify this, suppose that equation $(5.6)$ is wrong, from equations $(5.3)$ and $(5.4)$ we would infer that there exists $(s_{0},t_{0})\in (l,L)\times (0,T)$ such that 
		\begin{equation}
			\max_{(s,t)\in (l,L)\times (0,t_{0}]} z(s,t)=z(s_{0},t_{0})=0.
		\end{equation}
		Moreover, $z_{t}(s_{0},t_{0})\ge 0$ and $z_{s}(s_{0},t_{0})= 0$ can be derived. Furthermore, since $\underline{w}$ and $\overline{w}$ belong to $C^{2,1}((l,L)\times(0,T])$, $z_{ss}(s,t_{0})$ exists for all $s \in (l,L)$ and
		\begin{equation}
			z_{s}(s,t_{0})=\int_{s_{0}}^{s}z_{ss}(\sigma , t_{0})d\sigma \qquad for \ all \ s\in (l,L).
		\end{equation}
		Since $z(\cdot, t_{0})$ gets its maximum at $s_{0}$, it follows from equation $(5.8)$ that when $j \to \infty$, there exists $(s_{j})_{j\in N} \subset (s_{0}, L)$ such that $s_{j} \searrow s_{0}$, and 
		$$
			z_{ss}(s_{j}, t_{0}) \le 0 \qquad for \ all \ j \in N.
		$$
		Otherwise $(5.8)$ implies $z_{s}(s,t_{0})>0 \ for \ all \ s\in (s_{0}, L)$, which contradicts equation $(5.7)$. \par
		Since $z_{ss}(s_{j},t_{0})\le 0$ and $z_{ss}=\underline{w}_{ss}-\overline{w}_{ss}$, $\underline{w}_{ss}(s_{j},t_{0})\le \overline{w}_{ss}(s_{j}, t_{0})$.
		$$
			\begin{array}{ll}
				z_{t}(s_{j}, t_{0}) & =\underline{w}_{t}(s_{j}, t_{0})-\overline{w}_{t}(s_{j}, t_{0})-2C_{1}\varepsilon e^{2C_{1}t_{0}}\\
				\quad & \le bs_{j}^{\gamma}\underline{w}(s_{j}, t_{0})\underline{w}_{s}^{\alpha}(s_{j}, t_{0})+cs_{j}^{\delta}\underline{w}_{s}^{\alpha}(s_{j}, t_{0})+d\underline{w}_{s}^{\alpha}(s_{j}, t_{0}) \\
				\quad & \quad -\big[ bs_{j}^{\gamma}\overline{w}(s_{j}, t_{0})\overline{w}_{s}^{\alpha}(s_{j}, t_{0})+cs_{j}^{\delta}\overline{w}_{s}^{\alpha}(s_{j}, t_{0})+d\overline{w}_{s}^{\alpha}(s_{j}, t_{0})  \big] \\
				\quad & \quad -2C_{1}\varepsilon e^{2C_{1}t_{0}}\\
				\quad & = bs_{j}^{\gamma}\overline{w}_{s}^{\alpha}(s_{j}, t_{0})(\underline{w}(s_{j},t_{0})-\overline{w}(s_{j},t_{0}))\\
				\quad & \quad + \big[bs_{j}^{\gamma} \underline{w}(s_{j}, t_{0})+cs_{j}^{\delta}+d \big](\underline{w}_{s}^{\alpha}(s_{j}, t_{0})-\overline{w}_{s}^{\alpha}(s_{j}, t_{0}))\\
				\quad & \quad -2C_{1}\varepsilon e^{2C_{1}t_{0}}. 
			\end{array}
		$$
		\indent Since $\underline{w}$ and $\overline{w}$ belong to $C^{0}([l,L]\times[0,T])$, $\underline{w}$ and $\overline{w}$ are bounded.
		Therefore, there exists a positive constant $C_{2}$ such that
		$$
			\begin{array}{ll}
				z_{t}(s_{j}, t_{0}) & \le bs_{j}^{\gamma}\overline{w}_{s}^{\alpha}(s_{j}, t_{0})(\underline{w}(s_{j},t_{0})-\overline{w}(s_{j},t_{0}))\\
				\quad & \quad + C_{2}(\underline{w}_{s}^{\alpha}(s_{j}, t_{0})-\overline{w}_{s}^{\alpha}(s_{j}, t_{0})) -2C_{1}\varepsilon e^{2C_{1}t_{0}}.
			\end{array} 
		$$
		Since $\underline{w}$ and $\overline{w}$ belong to $C^{0}([l,L]\times[0,T])\cap C^{2,1}((l,L)\times(0,T])$, we can let $j \to \infty$ here and have $$\underline{w}_{s}(s_{0}, t_{0})=\overline{w}_{s}(s_{0}, t_{0})$$ due to $z_{s}(s_{0}, t_{0})=0$ and $z_{s}=\underline{w}_{s}-\overline{w}_{s}$. Therefore, we can get
		$$
			\begin{array}{ll}
				z_{t}(s_{0}, t_{0}) & \le bs_{0}^{\gamma}\overline{w}_{s}^{\alpha}(s_{0}, t_{0})(\underline{w}(s_{0},t_{0})-\overline{w}(s_{0},t_{0}))-2C_{1}\varepsilon e^{2C_{1}t_{0}}\\
				\quad & \le C_{1}(\underline{w}(s_{0},t_{0})-\overline{w}(s_{0},t_{0}))-2C_{1}\varepsilon e^{2C_{1}t_{0}}. 
			\end{array} 
		$$
		From $(5.7)$,  $z(s_{0},t_{0})=\underline{w}(s_{0},t_{0})-\overline{w}(s_{0},t_{0})-\varepsilon e^{2C_{1}t_{0}}=0$, so 
		$$
			z_{t}(s_{0}, t_{0})\le 0,
		$$
		which contradicts $(5.6)$. Letting $\varepsilon \searrow 0$, by $(5.3)$, $(5.4)$ and $(5.6)$, we obtain $(5.5)$. 
	\end{proof}

	\newpage
	\bibliographystyle{unsrt}
	\bibliography{reference}	

@article{1953Patlak,
	title={Random walk with persistence and external bias},
	author={ Patlak, Clifford S. },
	journal={The Bulletin of Mathematical Biophysics},
	volume={15},
	number={3},
	pages={311-338},
	year={1953},
}

@article{KELLER1971225,
	title={Model for chemotaxis},
	author={Evelyn F. Keller and Lee A. Segel},
	journal={Journal of Theoretical Biology},
	volume={30},
	number={2},
	pages={225-234},
	year={1971},
}

@article{Wu2005SignalingMF,
	title={Signaling mechanisms for regulation of chemotaxis},
	author={Dianqing Wu},
	journal={Cell Research},
	year={2005},
	volume={15},
	number={1},
	pages={52-56},
}

@article {PMID:16857884,
	title = {Free-flight responses of Drosophila melanogaster to attractive odors},
	author = {Budick, Seth A and Dickinson, Michael H},
	number = {15},
	volume = {209},
	year = {2006},
	journal = {The Journal of experimental biology},
	pages = {3001-3017},
}

@article{Kennedy1974PheromoneRegulatedAI,
	title={Pheromone-Regulated Anemotaxis in Flying Moths},
	author={J. S. Kennedy and David Marsh},
	journal={Science},
	year={1974},
	volume={184},
	number={4140}, 
	pages={1001-999},
}

@article{2001Finite,
	title={Finite dimensional attractor for one-dimensional Keller-Segel equations},
	author={ Osaki, Koichi  and  Yagi, Atsushi },
	journal={Funkcialaj Ekvacioj},
	volume={44},
	number={3},
	pages={441-469},
	year={2001},
}

@article{1997Application,
	title={Application of the Trudinger-Moser inequality to a parabolic system of chemotaxis},
	author={ Nagai, T.  and  Senba, T.  and  Yoshida, K. },
	journal={Funkc Ekvacioj},
	volume={40},
	number={3},
	pages={411-433},
	year={1997},
}

@article{Winkler2010AggregationVG,
	title={Aggregation vs. global diffusive behavior in the higher-dimensional Keller–Segel model},
	author={Michael Winkler},
	journal={Journal of Differential Equations},
	year={2010},
	number={12},
	volume={248},
	pages={2889-2905},
}

@article{Bellomo2015TowardAM,
	title={Toward a mathematical theory of Keller–Segel models of pattern formation in biological tissues},
	author={Nicola Bellomo and Nicola Bellomo and Abdelghani Bellouquid and Youshan Tao and Michael Winkler},
	journal={Mathematical Models and Methods in Applied Sciences},
	year={2015},
	number={9},
	volume={25},
	pages={1663-1763},
}

@article{Hillen2009AUG,
	title={A user’s guide to PDE models for chemotaxis},
	author={Thomas Hillen and Kevin J. Painter},
	journal={Journal of Mathematical Biology},
	year={2009},
	number={1-2},
	volume={58},
	pages={183-217},
}

@article{2003From,
	title={From 1970 until present: The Keller-Segel model in chemotaxis and its consequences I},
	author={ Horstmann, D. },
	journal={Jahresbericht der Deutschen Mathematiker-Vereinigung},
	volume={105},
	number={3},
	pages={103-165},
	year={2003},
}

@article{nagai2001blowup,
	title={Blowup of nonradial solutions to parabolic--elliptic systems modeling chemotaxis in two-dimensional domains},
	author={Nagai, Toshitaka},
	journal={Journal of Inequalities and Applications},
	volume={2001},
	number={1},
	pages={37-55},
	year={2001},
}

@article{horstmann2001blow,
	title={Blow-up in a chemotaxis model without symmetry assumptions},
	author={Horstmann, Dirk and Wang, Guofang},
	journal={European Journal of Applied Mathematics},
	volume={12},
	number={2},
	pages={159-177},
	year={2001},
}

@article{jager1992explosions,
	title={On explosions of solutions to a system of partial differential equations modelling chemotaxis},
	author={J{\"a}ger, Willi and Luckhaus, Stephan},
	journal={Transactions of the american mathematical society},
	volume={329},
	number={2},
	pages={819-824},
	year={1992},
}

@article{winkler2019unstable,
	title={How unstable is spatial homogeneity in Keller-Segel systems? A new critical mass phenomenon in two-and higher-dimensional parabolic-elliptic cases},
	author={Winkler, Michael},
	journal={Mathematische Annalen},
	volume={373},
	number={10},
	pages={1237-1282},
	year={2019},
}

@article{winkler2010boundedness,
	title={Boundedness and finite-time collapse in a chemotaxis system with volume-filling effect},
	author={Winkler, Michael and Djie, Kianhwa C},
	journal={Nonlinear Analysis: Theory, Methods \& Applications},
	volume={72},
	number={2},
	pages={1044-1064},
	year={2010},
}

@article{winkler2024slow,
	title={Slow Grow-up in a Quasilinear Keller--Segel System},
	author={Winkler, Michael},
	journal={Journal of Dynamics and Differential Equations},
	volume={36},
	number={2},
	pages={1677-1702},
	year={2024},
}

@article{winkler2024complete,
	title={Complete infinite-time mass aggregation in a quasilinear Keller--Segel system},
	author={Winkler, Michael},
	journal={Israel Journal of Mathematics},
	volume={263},
	number={1},
	pages={93-127},
	year={2024},
}

@article{tu2022effects,
	title={On effects of the nonlinear signal production to the boundedness and finite-time blow-up in a flux-limited chemotaxis model},
	author={Tu, Xinyu and Mu, Chunlai and Zheng, Pan},
	journal={Mathematical Models and Methods in Applied Sciences},
	volume={32},
	number={4},
	pages={647-711},
	year={2022},
}

@article{chiyoda2020finite,
	title={Finite-time blow-up in a quasilinear degenerate chemotaxis system with flux limitation},
	author={Chiyoda, Yuka and Mizukami, Masaaki and Yokota, Tomomi},
	journal={Acta Applicandae Mathematicae},
	volume={167},
	number={1},
	pages={231-259},
	year={2020},
}

@article{fluchter2024solutions,
	title={Solutions to a chemotaxis system with spatially heterogeneous diffusion sensitivity},
	author={Fl{\"u}chter, Gregor},
	journal = {Journal of Differential Equations},
	volume = {424},
	pages = {1-29},
	year = {2025},
}

@article{winkler2011blow,
	title={Blow-up in a higher-dimensional chemotaxis system despite logistic growth restriction},
	author={Winkler, Michael},
	journal={Journal of Mathematical Analysis and Applications},
	volume={384},
	number={2},
	pages={261-272},
	year={2011},
}

@article{cieslak2008finite,
	title={Finite-time blow-up in a quasilinear system of chemotaxis},
	author={Cie{\'s}lak, Tomasz and Winkler, Michael},
	journal={Nonlinearity},
	volume={21},
	number={5},
	pages={1057-1076},
	year={2008},
}

@article{winkler2010does,
	title={Does a'volume-filling effect'always prevent chemotactic collapse?},
	author={Winkler, Michael},
	journal={Mathematical Methods in the Applied Sciences},
	volume={33},
	number={1},
	pages={12-24},
	year={2010},
}

@article{horstmann2005boundedness,
	title={Boundedness vs. blow-up in a chemotaxis system},
	author={Horstmann, Dirk and Winkler, Michael},
	journal={Journal of Differential Equations},
	volume={215},
	number={1},
	pages={52-107},
	year={2005},
}

@article{herrero1997blow,
	title={A blow-up mechanism for a chemotaxis model},
	author={Herrero, Miguel A and Vel{\'a}zquez, Juan JL},
	journal={Annali della Scuola Normale Superiore di Pisa-Classe di Scienze},
	volume={24},
	number={4},
	pages={633-683},
	year={1997},
}

@article{cieslak2010finite,
	title={Finite time blow-up for a one-dimensional quasilinear parabolic--parabolic chemotaxis system},
	author={Cie{\'s}lak, Tomasz and Lauren{\c{c}}ot, Philippe},
	journal={Annales de l'Institut Henri Poincar{\'e} C, Analyse non lin{\'e}aire},
	volume={27},
	number={1},
	pages={437-446},
	year={2010},
}

@article{winkler2022approaching,
	title={Approaching logarithmic singularities in quasilinear chemotaxis-consumption systems with signal-dependent sensitivities.},
	author={Winkler, Michael},
	journal={Discrete \& Continuous Dynamical Systems-Series B},
	volume={27},
	number={11},
	pages={6565-6587},
	year={2022},
}

@article{winkler2024bounds,
	title={L$^{\infty}$ bounds in a two-dimensional doubly degenerate nutrient taxis system with general cross-diffusive flux},
	author={Winkler, Michael},
	journal={Journal of Differential Equations},
	volume={400},
	number={10},
	pages={423-456},
	year={2024},
}

@article{bellomo2017finite,
	title={Finite-time blow-up in a degenerate chemotaxis system with flux limitation},
	author={Bellomo, Nicola and Winkler, Michael},
	journal={Transactions of the American mathematical society, Series B},
	volume={4},
	number={2},
	pages={31-67},
	year={2017},
}

@article{article2000,
	author = {Larrivée, B and Karsan, A},
	year = {2000},
	month = {06},
	pages = {447-56},
	title = {Signaling pathways induced by vascular endothelial growth factor},
	volume = {5},
	journal = {International journal of molecular medicine},
}

@article{LIN2018435,
	title = {A blow-up result for a quasilinear chemotaxis system with logistic source in higher dimensions},
	journal = {Journal of Mathematical Analysis and Applications},
	volume = {464},
	number = {1},
	pages = {435-455},
	year = {2018},
	issn = {0022-247X},
	author = {Ke Lin and Chunlai Mu and Hua Zhong},
}
\end{document}